\documentclass[11pt]{article}

\usepackage{amsfonts,amsmath,amsthm,amssymb,enumerate,color}

\allowdisplaybreaks 

\usepackage[makeroom]{cancel}
\usepackage{float}
\usepackage{authblk}
\usepackage{bbm}
\usepackage{array}
\usepackage{mathtools}
\usepackage{graphicx}
\usepackage{hyperref}

\numberwithin{equation}{section}
\newtheorem{theorem}{Theorem}[section]
\newtheorem{definition}[theorem]{Definition}
\newtheorem{proposition}[theorem]{Proposition}
\newtheorem{corollary}[theorem]{Corollary}
\newtheorem{lemma}[theorem]{Lemma}
\newtheorem{remark}[theorem]{Remark}
\newtheorem{example}[theorem]{Example}

\newenvironment{Example}{\begin{example}\rm}{\end{example}}
\newenvironment{Remark}{\begin{remark}\rm}{\end{remark}}

\newcommand{\R}{\mathbb{R}}

\newcommand{\M}{\mathbb{M}}
\newcommand{\Ho}{\mathcal{H}}
\newcommand{\V}{\mathcal V}
\newcommand{\ve}{\varepsilon}

\newcommand{\B}{\mathbb{B}}
\newcommand{\Sph}{\mathbb{S}}
\newcommand{\AdS}{\mathbf{AdS}}

\newcommand{\U}{\mathbf{U}}
\newcommand{\SO}{\mathbf{SO}}
\newcommand{\fso}{\mathfrak{so}}
\newcommand{\SU}{\mathbf{SU}}
\newcommand{\Sp}{\mathbf{Sp}}
\newcommand{\Spin}{\mathbf{Spin}}

\newcommand{\Comp}{\mathbb{C}}
\newcommand{\Quat}{\mathbb{H}}
\newcommand{\Octo}{\mathbb{O}}

\newcommand{\proj}{\mathrm{P}}
\newcommand{\RP}{\R\proj}
\newcommand{\CP}{\Comp\proj}

\newcommand{\OP}{\Octo\proj}

\newcommand{\hypb}{\mathrm{H}}

\newcommand{\CH}{\Comp\hypb}

\newcommand{\OH}{\Octo\hypb}

\newcommand{\calC}{\mathcal{C}}
\newcommand{\Cl}{\mathbf{Cl}}
\newcommand{\Dh}{{\Delta}_{\Ho}}
\newcommand{\diam}{\mathbf{diam}}
\newcommand{\End}{\mathbf{End}}
\newcommand{\fa}{\mathfrak{a}}
\newcommand{\fA}{\mathfrak{A}}
\newcommand{\Id}{\mathbf{Id}}
\newcommand{\Lie}{\mathcal{L}}
\newcommand{\rank}{\mathrm{rank}}
\newcommand{\ri}{\mathbf{Ric}}
\newcommand{\s}[1]{\Gamma(#1)}
\newcommand{\spn}{\mathrm{span}}
\newcommand{\tr}{\mathrm{Tr}}

\mathtoolsset{showonlyrefs=true}

\newlength{\bibitemsep}
\newlength{\bibparskip}
\let\oldthebibliography\thebibliography
\renewcommand\thebibliography[1]{
  \oldthebibliography{#1}
  \setlength{\parskip}{\bibitemsep}
  \setlength{\itemsep}{\bibparskip}
}

\newcommand{\Addresses}{{
  \bigskip
  \footnotesize

  F.~Baudoin, \textsc{University of Connecticut, Department of Mathematics, 341 Mansfield Road U1009, 06269 Storrs, CT, USA}\par\nopagebreak
  \textit{E-mail address}, F.~Baudoin: \texttt{fabrice.baudoin@uconn.edu}

  \medskip

  E.~Grong, \textsc{University of Bergen, Department of Mathematics, P.O. Box 7803, 5020 Bergen, Norway}\par\nopagebreak
  \textit{E-mail address}, E.~Grong: \texttt{erlend.grong@uib.no}

 \medskip
 
  L.~Rizzi, \textsc{SISSA, via Bonomea 265, 34136 Trieste, Italy}\par\nopagebreak
  \textit{E-mail address}, L.~Rizzi: \texttt{luca.rizzi@sissa.it}

\medskip

  G.~Vega-Molino, \textsc{University of Bergen, Department of Mathematics, P.O. Box 7803, 5020 Bergen, Norway}\par\nopagebreak
  \textit{E-mail address}, G.~Vega-Molino: \texttt{gianmarco.vega-molino@uib.no}

}}

\title{H-type Foliations}

\author{Fabrice Baudoin\footnote{Supported in part by NSF grant DMS-1660031 and the Simons Foundation grant 586355}, Erlend Grong\footnote{Supported by project 249980/F20 of the Norwegian Research Council}, Luca Rizzi\footnote{Supported by the Grant ANR-15-CE40-0018 of the ANR, by the ANR project ANR-15-IDEX-02, and from the European Research Council (ERC) under the European Union's Horizon 2020 research and innovation programme (grant agreement No. 945655)}, Gianmarco Vega-Molino}

\begin{document}

\maketitle

\begin{abstract}
With a view toward sub-Riemannian geometry, we introduce and study   H-type foliations. These structures are natural generalizations of K-contact geometries which encompass as special cases K-contact manifolds, twistor spaces, 3K-contact manifolds and H-type groups.  Under an horizontal Ricci curvature lower bound on these structures, we prove a sub-Riemannian diameter upper bounds and first eigenvalue estimates for the sub-Laplacian. Then, using  a result by Moroianu-Semmelmann \cite{MS}, we  classify the H-type foliations that carry  a parallel horizontal Clifford structure. Finally, we prove an horizontal Einstein property and compute the horizontal Ricci curvature of these spaces in codimension more than $2$. 
\end{abstract}

\tableofcontents

\section{Introduction}

\subsection{Motivation}

A sub-Riemannian manifold is a smooth  manifold $\M$ equipped with a bracket-generating distribution $\Ho \subset T\M$ and a fiber inner product $g_\Ho$ on $\Ho$. The distribution $\Ho $ is referred to as the horizontal distribution. The bracket-generating condition means  that if we denote by $L(\Ho)$ the Lie algebra of the vector fields generated by the global $C^\infty$ sections of $\Ho$, then $\text{span} \{X(x)\mid X\in L(\Ho)\} = T_x \M$ for every $x\in \M$. Broadly speaking, sub-Riemannian geometry is the study of the intrinsic properties of the triple  $(\M,\Ho,g_\Ho)$. Sub-Riemannian geometry is at the interface of many fields, including: geometric control theory, metric geometry, analysis of subelliptic partial differential equations, stochastic analysis and Riemannian geometry. As such, it has been studied, possibly under different names, from many different viewpoints. To get an overview of this rich and vibrant subject, one may  consult the monographs \cite{ABB}, \cite{Ag-Sa}, \cite{Gromov},  \cite{Montgomery},  or \cite{Rifford}.

The purpose of this paper is to introduce and study a new class of sub-Riemannian manifolds generalizing the H-type groups introduced by Kaplan in \cite{Kaplan}. We call such manifolds H-type sub-Riemannian manifolds. Due to their symmetries, H-type sub-Riemannian manifolds provide an ideal framework to develop a  program reducing the study of global geometric, metric, or analytic properties of the ambient space to the study of local sub-Riemannian curvature type invariants. This geometric analysis program will be further developed in a subsequent work. In the present paper, we study H-type sub-Riemannian manifolds arising from a special type of totally geodesic foliations, which we will refer to as H-type foliations. Roughly speaking, H-type foliations concern a special case of Riemannian manifolds $(\M,g)$ that are foliated transversely to a sub-Riemannian structure. We will write this data as $(\M,\Ho,g)$, where the bracket-generating distribution $\Ho$ has constant rank and crucially its complement $\V = \Ho^\perp$ is integrable and tangent to the foliation. At each point $p \in \M$, one has a representation denoted by $J$ of the Clifford algebra $\Cl(\V_p)$ onto the space of horizontal endomorphisms $\Ho_p \to \Ho_p$. We will call the sub-Riemannian manifold $\M, \Ho, g_\Ho)$ obtained by restricting the metric $g$ to $\Ho$ an H-type sub-Riemannian manifold. In the case of an H-type group, the complement $\V$ is given by the center of the group. In the case of a regular K-contact or 3K-contact  structure, this complement is determined by respectively the orbits of a $\U(1)$ or a $\SO(3)$ isometric action on $\M$. 

The main motivation that led to the construction and study H-type sub-Riemannian manifolds was the desire to provide a unified framework for many results obtained in the last few years in the geometric analysis of sub-Riemannian manifolds through different techniques (see for instance \cite{BRcontact},  \cite{BI17}, \cite{BG17}, \cite{RS-3sas}). The interest of H-type foliations as model spaces in sub-Riemannian geometry is demonstrated in Section \ref{section CD} of the present paper, where we show that on H-type foliations the generalized curvature dimension inequality introduced in \cite{BG17} can be controlled using information from only the horizontal Ricci curvature of the Bott connection.  Some consequences of this fact are pointed out in Corollary \ref{yang-mills corollary}, but we refer to the survey \cite{BaudoinEMS2014} for many other known consequences of the generalized curvature dimension inequality. In the subsequent paper \cite{BGMR19}, we show that the techniques developed in \cite{BGKT17} extend to H-type sub-Riemannian manifolds as well, and as consequence we will obtain for those structures sharp Bonnet-Myers theorem and  sharp sub-Laplacian comparison theorems.

\subsection{Main results}

A first highlight of the paper is Theorem \ref{th:AllYM}, where we prove that H-type foliations are necessarily Yang-Mills. As a consequence,  the sub-Laplacian of an H-type foliation satisfies a simple Bochner's type formula and the validity of the generalized curvature dimension inequality depends only on horizontal Ricci curvature, see Proposition \ref{generalized CD}. Applications of generalized curvature dimension inequalities in sub-Riemannian geometry have extensively been studied in the last few years (see \cite{BBG14, BG17,GrTh16a}) and, in the present setting, some corollaries are pointed out in Corollaries  \ref{yang-mills corollary} and \ref{compact finite fundamental corollary}. In particular, sub-Riemannian diameter upper bounds and first eigenvalue estimates for the sub-Laplacian are obtained.

A second  highlight of the paper is the classification of H-type foliations that carry a parallel horizontal Clifford structure. Roughly speaking, from Theorem \ref{Formula psi},  an H-type foliation carries a parallel horizontal Clifford structure if $\nabla_\Ho J=0$ and for all vertical vectors $u,v \in \V$,
$$
(\nabla_{u} J)_{v} =J_{\Psi(u,v)},
$$
with $\Psi(u,v)=-\kappa (u \cdot v +\langle u ,v \rangle )$, where $J$ is the representation of the Clifford algebra $\Cl(\V)$ on the space of  endomorphisms of  $\Ho $, $\nabla$ the Bott connection of the foliation (see Section \ref{s:setting}),  and $\kappa$ is a constant such that $\kappa^2$ is the sectional curvature of the leaves of the foliation.  

In the influential paper \cite{MS},  A. Moroianu and U. Semmelmann introduced the related  concept of parallel even Clifford structures on Riemannian manifolds. In a sense, for $n \neq 8$,   H-type foliations with a parallel horizontal Clifford structure are  to parallel even Clifford structures on Riemannian manifolds what Sasakian and 3-Sasakian manifolds are respectively to K\"ahler and quaternion K\"ahler manifolds; see Corollary \ref{parallel even} for a precise statement.  We show in Theorem \ref{horizontal einstein} that H-type foliations with a parallel horizontal Clifford structure are always  horizontally Einstein if the rank of $\V$ is greater or equal than 2 and different from 3. More precisely, for $m \ge 2$, $m \neq 3$,  we prove that one has
\begin{equation}
\ri_\Ho= \frac{\kappa}{4} (n + 8(m - 1) ) g_\Ho,
\end{equation}
where $n$ is the rank of $\Ho$, $m$ the rank of $\V$ and $\ri_\Ho$ the horizontal Ricci curvature of the Bott connection. The case $m=3$ is special, due to the Lie algebra splitting $\fso(4)=\fso(3)\oplus\fso(3)$,  and we prove in that case that if the structure is of quaternionic type then 
\begin{equation}
\ri_\Ho= \frac{\kappa}{2} (n + 8 ) g_\Ho.
\end{equation}
 In  Theorem \ref{t:curvature constancy1} we prove  that if  $\kappa \neq 0$, then the vertical distribution $\V$ of a H-type foliation $(\M, \Ho,g)$ lies in the curvature constancy (in the sense of Gray \cite{Gray}) of the metric $$\hat{g} = g_\Ho \oplus 2\kappa g_\V ,$$ where $g_\Ho$ and $g_\V$ respectively denote the projections of the original Riemannian metric $g$ on $\Ho$ and $\V$. Interestingly, we note that if $\kappa >0$, $\hat{g}$ is a Riemannian metric, whereas if $\kappa <0$ then $\hat{g}$ is a semi-Riemannian metric.  From Theorem 3.7 in  \cite{MS} (which describes all the Riemannian submersions with non-trivial curvature constancy) one deduces therefore the complete classification of complete simply connected H-type  foliations with a parallel horizontal Clifford structure coming from a globally defined submersion $\pi :\M \to \B$ and $\kappa \neq 0$. Those submersions are described in the following two tables. In Table \ref{Table 1}, the H-type foliation  is obtained from a totally geodesic Riemannian submersion $\pi : \M \to \B$ whose fibers coincide with the curvature constancy of $\M$. Notations, conventions and terminology are standard, but for further details we refer to \cite{MS} from which this table taken. In particular, for $n=4$ the quaternion-K\"ahler property is understood in the sense that $\B$ is Einstein and anti self-dual (see \cite{Besse}, Chapter 13). In Table \ref{Table 2}, the H-type foliation  is obtained from a totally geodesic semi-Riemannian submersion $\pi : \M \to \B$ whose fibers are in the curvature constancy of $\M$.  We note that Table \ref{Table 1} contains the quaternionic and octonionic Hopf fibrations  and that Table \ref{Table 2} contains the quaternionic and octonionic anti-de Sitter fibrations.  We  also note that all the examples of $\M$ in Table \ref{Table 1}  are compact.

 \begin{table}[H]
\centering
\scalebox{0.8}{
\begin{tabular}{|p{3.5cm}||p{5cm}|c|>{\centering\arraybackslash}p{1.8cm}|c|  }
  \hline
 $\M$ &   $\B $  & Fiber  & $\rank (\Ho)$ & $\rank (\V)$ \\
 \hline
 \hline
Twistor space & Quaternion-K\"ahler with positive scalar curvature& $\Sph^2$ & $4k$ & $2$ \\
3-Sasakian &Quaternion-K\"ahler with positive scalar curvature & $\Sph^3$ & $4k$ & $3$ \\
Quaternion-Sasakian  & Product of two quaternion-K\"ahler with positive scalar curvature & $\RP^3 $ & $4k$ & $3$ \\
$\frac{\Sp(q^++1)\times \Sp(q^-+1)}{\Sp(q^+) \times \Sp(q^-)\times \Sp(1) }$ & $\Quat P^{q^+} \times \Quat P^{q^-}$ & $\Sph^3$ & $4(q^++q^-)$ & $3$ \\
  \hline
 \hline 
$\frac{\Sp(k+2)}{\Sp(k) \times \Spin(4) }$ & $\frac{\Sp(k+2)}{\Sp(k)\times \Sp(2)}$ & $\Sph^4$ & $8k$ & $4$  \\
$\frac{\SU(k+4)}{\Sph ( \U(k) \times \Sp(2)\U(1)  )}$ & $\frac{\SU(k+4)}{ \Sph ( \U(k) \times \U(4))} $ & $\RP^5$ & $8k$ & $5$ \\ 
$\frac{\SO(k+8)}{\SO(k) \times \Spin(7) }$ & $\frac{\SO(k+8)}{\SO(k)\times \SO(8)}$ & $\RP^7$ & $8k$, $k \ge 3$, $k$ odd & $7$ \\
$\frac{\Spin(k+8)}{\SO(k) \times \Spin(7) }$ & $\frac{\SO(k+8)}{\SO(k)\times \SO(8)}$ & $\Sph^7$ & $8k$, $k = 1$, $k$ even & $7$ \\
  \hline
 \multicolumn{5}{|c|}{Exceptional cases} \\
 \hline
$\frac{F_4}{\Spin(8)}$  &$\frac{F_4}{\Spin(9)}=\OP^2$ & $\Sph^8$& $16$ & $8$\\
$\frac{E_6}{\Spin(8)\U(1)}$ & $\frac{E_6}{\Spin(10) \U(1)}=(\Comp \otimes \Octo)P^2$ & $\S^9$& $32$ & $9$\\
$\frac{E_7}{\Spin(11)\SU(2)}$ & $\frac{E_7}{\Spin(12) \SU(2)}=(\Quat \otimes \Octo) P^2$ & $\Sph^{11}$& $ 64$ & $11$ \\
$\frac{E_8}{\Spin(15)}$ & $\frac{E_8}{\Spin^+(16)} =(\Octo \otimes \Octo) P^2$ & $\Sph^{15}$& $128$ & $ 15$ \\
  \hline
\end{tabular}}
\caption{H-type submersions with a parallel horizontal Clifford structure and $\kappa >0$.}
\label{Table 1}
\end{table}

\begin{table}[H]
\centering
\scalebox{0.8}{
\begin{tabular}{|p{3.5cm}||p{5cm}|c|>{\centering\arraybackslash}p{1.8cm}|c|  }
  \hline
 $\M$ &   $\B $  & Fiber  & $\rank (\Ho)$ &$\rank (\V)$ \\
 \hline
 \hline
Negative Twistor space & Quaternion-K\"ahler with negative scalar curvature& $\Sph^2$ & $4k$ & $2$ \\
Negative 3-Sasakian &Quaternion-K\"ahler with negative scalar curvature& $\Sph^3$ & $4k$ & $3$ \\
Negative Quaternion-Sasakian  & Product of two quaternion-K\"ahler with negative scalar curvature  & $\RP^3 $ & $4k$ & $3$\\
$\frac{\Sp(q^+,1)\times \Sp(q^-,1)}{\Sp(q^+) \times \Sp(q^-)\times \Sp(1) }$ & $\Quat H^{q^+} \times \Quat H^{q^-}$ & $\Sph^3$ & $4(q^++q^-)$ & $3$ \\
 \hline
 \hline 
$\frac{\Sp(k,2)}{\Sp(k) \times \Spin(4) }$ & $\frac{\Sp(k,2)}{\Sp(k)\times \Sp(2)}$ & $\Sph^4$ & $8k$ & $4$  \\
$\frac{\SU(k,4)}{\Sph ( \U(k) \times \Sp(2)\U(1)  )}$ & $\frac{\SU(k,4)}{ \Sph ( \U(k) \times \U(4))} $ & $\RP^5$ & $8k$ & $5$ \\ 
$\frac{\SO(k,8)}{\SO(k) \times \Spin(7) }$ & $\frac{\SO(k,8)}{\SO(k)\times \SO(8)}$ & $\RP^7$ & $8k$, $k \ge 3$, $k$ odd & $7$ \\
$\frac{\Spin(k,8)}{\SO(k) \times \Spin(7) }$ & $\frac{\SO(k,8)}{\SO(k)\times \SO(8)}$ & $\Sph^7$ & $8k$, $k = 1$, $k$ even & 7 \\
  \hline
 \multicolumn{5}{|c|}{Exceptional cases} \\
 \hline
$\frac{F_4^{-20}}{\Spin(8)}$  &$\frac{F^{-20}_4}{\Spin(9)}=\OH^2$ & $\Sph^8$& $ 16$ & $8$ \\
$\frac{E_6^{-14}}{\Spin(8)\U(1)}$ & $\frac{E^{-14}_6}{\Spin(10) \U(1)}=(\Comp \otimes \Octo)H^2$ & $\Sph^9$& $32$ & $9$\\
$\frac{E_7^{-5}}{\Spin(11)\SU(2)}$ & $\frac{E_7^{-5}}{\Spin(12) \SU(2)}=(\Quat \otimes \Octo)H^2$ & $\Sph^{11}$& $64$ & $11$\\
$\frac{E^8_8}{\Spin(15)}$ & $\frac{E^8_8}{\Spin^+(16)} =(\Octo \otimes \Octo)H^2$ & $\Sph^{15}$& $128$ & $15$\\
  \hline
\end{tabular}}
\caption{H-type submersions   with a parallel horizontal Clifford structure and $\kappa <0$.}
\label{Table 2}
\end{table}

 The case $\kappa=0$ is special. It corresponds to H-type foliations for which the leaves are flat. Such foliations are described in Theorem \ref{flat leaves} where we prove that if the foliation comes from a totally geodesic submersion, then the base space of that submersion is K\"ahler for $m=1$, locally hyper-K\"ahler for $m=2$ or $m=3$ and flat for $m \ge 4$. We conclude the paper with several estimates on  sub-Riemannian diameter and   first eigenvalue of the sub-Laplacian.
   
\section{H-type foliations}

\subsection{Totally geodesic foliations }\label{s:setting}

Let $(\M,g)$ be a smooth, oriented, connected, Riemannian manifold with dimension $n+m$. For notational simplicity, for $X,Y \in \s\M$ we will often denote $\langle X, Y \rangle=g(X,Y)$. We assume that $\M$ is equipped with a Riemannian foliation with bundle-like complete metric $g$ and totally geodesic  $m$-dimensional leaves. The sub-bundle $\V$ formed by vectors tangent to the leaves is referred  to as the set of \emph{vertical directions}. The sub-bundle $\Ho$ which is normal to $\V$ is referred to as the set of \emph{horizontal directions}. If one denotes by $\Lie$ the Lie derivative, from  Theorem 5.19, p. 56 in \cite{Tondeur}, the bundle-like property for a Riemannian foliation is equivalent to the fact that for every $X \in \s\Ho$, $Z \in \Gamma (\V)$,
 \begin{equation}
 \Lie_Z g (X,X)=0,
 \end{equation}
 and from Theorem 5.23, p. 58 in \cite{Tondeur}, the totally geodesic foliation property is equivalent to the fact that for every $X \in \s\Ho$, $Z \in \Gamma (\V)$,
 \begin{equation}
 \Lie_X g (Z,Z)=0.
 \end{equation}
We simply refer to these structures as \emph{totally geodesic foliations} (the bundle-like property of the metric is always assumed in this paper).  We refer to \cite{BaudoinEMS2014, Tondeur} and references therein for details about the geometry of totally geodesic  foliations. For later reference, we note that those definitions extend to the case where $g$ is semi-Riemannian.

\paragraph{Important:} From now on, unless stated otherwise, we will always assume that $(\M,g)$ is a foliated Riemannian manifold such that the foliation is \emph{both} Riemannian and totally geodesic. We let $\V$ be the subbundle tangent to the leaves, with an orthogonal complement $\Ho$ which we assume is bracket-generating. We will denote this structure as $(\M, \Ho, g)$. 

\

 Preliminary examples of such structures include the following.
\begin{Example}(K-contact manifolds)
 Let $(\M,\theta)$ be a $2n+1$-dimensional smooth contact manifold with Reeb vector field $Z$. The Reeb foliation on $\M$ is given by the orbits of $Z$.  From  \cite{Sasaki60}, it is always possible to find a Riemannian metric $g$ and a $(1,1)$-tensor field $J$ on $\M$ so that for all vector fields $X, Y$
\begin{equation} \label{JmapKCont}
g(X,Z)=\theta(X),\quad J^2(X)=-X+\theta (X) Z, \quad 2g(JX,Y)= d\theta(X,Y).
\end{equation}
The triple $(\M, \theta,g)$ is called a contact Riemannian manifold. The Reeb foliation is totally geodesic with bundle like metric if and only if the Reeb vector field $Z$ is a Killing field. In that case, $(\M, \theta,g)$ is called a K-contact Riemannian manifold. Observe that the horizontal distribution $\Ho$ is then the kernel of $\theta$ and that $\Ho$ is bracket-generating because $\theta$ is a contact form. Sasakian manifolds are the $K$-contact manifolds for which $J$ is integrable (i.e. has a vanishing Nijenhuis tensor); see \cite{BoG} for further details on Sasakian foliations.
\end{Example}

\begin{Example}(Positive and negative 3K-contact manifolds)
Consider a smooth $(4n+3)$-dimensional Riemannian manifold $(\M,g)$, admitting three distinct K-contact structures i.e.\ non-degenerate one-forms $\theta_\alpha$, for $\alpha=1,2,3$ such that $(\M,\theta_\alpha,g)$ is a contact Riemannian manifold and each Reeb vector field $Z_\alpha$ is Killing for the Riemannian metric $g$. Define $J_\alpha$ relative to the contact structure $(\M, \theta_\alpha,g)$ as in \eqref{JmapKCont}. Furthermore, we assume that
\begin{equation}\label{eq:conditions3K}
(a)\quad g(Z_\alpha,Z_\beta) = \delta_{\alpha\beta}, \qquad (b)\quad [Z_\alpha,Z_\beta] = 2\epsilon_{\alpha\beta\gamma} Z_\gamma,
\end{equation}
where $\epsilon_{\alpha\beta\gamma}$ denotes the Levi-Civita symbol. Following \cite{Je} (see also \cite{Ishihara73, Tanno2}), we call $(\M,g)$ a 3K-contact (resp.\ negative 3K-contact) structure if for distinct $\alpha,\beta,\gamma$ it holds
\begin{equation}
J_\alpha J_\beta =\epsilon_{\alpha\beta\gamma} J_\gamma, \qquad (\text{resp. } J_\alpha J_\beta =-\epsilon_{\alpha\beta\gamma} J_\gamma).
\end{equation}
The bundle generated by the Reeb vector fields $\V = \spn\{Z_1,Z_2,Z_3\}$ is integrable and, thanks to the Killing condition, the leaves of the corresponding foliation are totally geodesic with bundle like metric. Therefore, letting
\begin{equation}
\Ho = \bigcap_{\alpha} \ker\theta_\alpha,
\end{equation}
we have $T_p \M = \Ho_p \oplus \V_p$, with $\Ho \perp \V$, and $(\M,\Ho,g)$ is a totally geodesic foliation.
\end{Example}

\begin{remark}[Notation for the foliation] \label{re:notation}
For a foliated Riemannian manifold, it is admittedly a bit unconventional to denote the foliation by the transverse bundle $\Ho$ rather than by bundle $\V$ tangent to the foliation or the collection of leaves of the foliation itself. However, since the $\Ho$ and $\V$ determine each other through the Riemannian metric $g$, and since much of our investigation is related to the sub-Riemannian manifold $(\M, \Ho, g|_{\Ho})$, we permit this slight abuse of notation.
\end{remark}

\begin{remark}[Bracket-generating condition] \label{re:Constant}
If $\Ho$ is bracket-generating subbundle of $T\M$, then by the Chow-Rashevskii theorem, any pair of points can be connected by a curve tangent to $\Ho$. In particular, this means that any function whose derivatives are zero in  the directions of $\Ho$, has to be constant. Furthermore, if $\nabla$ is any connection, then any tensor that is $\nabla$-parallel in the directions of $\Ho$ is uniquely determined by its value at one point.
\end{remark}

\subsection{The Bott connection}
There is a canonical connection on $\M$ that preserves the metric and the foliation structure (see~\cite{BaudoinEMS2014} and Chapter 5 in~\cite{Tondeur}), the Bott connection. It is uniquely characterized by the following proposition which is a special case of Lemma 2.13 in \cite{Hladky}.

\begin{proposition}
Let $(\M,\Ho,g)$ be a totally geodesic, Riemannian foliation with vertical bundle $\V$. There exists a unique metric connection $\nabla$ on $\M$, called the Bott connection of the foliation, such that:
\begin{itemize}
\item $\Ho$ and $\V$ are $\nabla$-parallel, i.e. for every $X \in \s\Ho$, $Y \in \s{T\M}$ and $Z \in \Gamma (\V)$,
\begin{equation}
\nabla_Y X \in \s\Ho, \quad \nabla_Y Z \in \s\V;
\end{equation}
\item The torsion $T$ of $\nabla$ satisfies 
\begin{equation}
T(\Ho,\Ho) \subset \V, \qquad T(\Ho,\V)=0, \qquad T(\V,\V) =0.
\end{equation}
\end{itemize}
\end{proposition}

More explicitly, the Bott connection is given as follows:

\begin{equation}
\nabla_X Y =
\begin{cases}
\pi_{\Ho} ( \nabla_X^g Y) & X,Y \in \Gamma (\Ho), \\
\pi_{\Ho} ( [X,Y])   & X \in \Gamma (\V), Y \in \Gamma (\Ho), \\
\pi_{\V} ( [X,Y])   &X \in \Gamma (\Ho), Y \in \Gamma (\V), \\
\pi_{\V} ( \nabla_X^g Y) & X,Y \in \Gamma (\V),
\end{cases}
\end{equation}
where $\nabla^g$ is the Levi-Civita connection of the metric $g$ and $\pi_\Ho$ (resp. $\pi_\V$) the projection on $\Ho$ (resp. $\V$). It is easy to check that  the Bott connection has a torsion which is given by:
\begin{equation}
T(X,Y)=-\pi_\V ([\pi_\Ho X, \pi_\Ho Y]).
\end{equation}
Then, for $Z \in \s\V$, there is a  unique skew-symmetric fiber endomorphism  $J_Z:\s\Ho \to \s\Ho$ such that for all horizontal vector fields $X$ and $Y$,
\begin{align}\label{Jmap}
g_\Ho (J_Z X,Y)= g_\V (Z,T(X,Y)),
\end{align}
where $T$ is the torsion tensor of $\nabla$. We then extend $J_{Z}$ to be $0$ on  $\s\V$. Also, if $Z\in \s\Ho$, from \eqref{Jmap} we set $J_Z=0$. 

\begin{Example}
Let $(\M, \theta,g)$ be a K-contact Riemannian manifold. The Bott connection coincides  with Tanno's connection that was introduced in \cite{Tanno}. In the case where $(\M, \theta,g)$ is Sasakian, the Bott connection coincides  with the Tanaka-Webster connection. 
\end{Example}

\begin{Example}
Let $(\M, \theta_1,\theta_2,\theta_3, g)$ be a 3K-contact Riemannian manifold with dimension strictly greater than 7. The Bott connection coincides then  with the Biquard connection. See Section 1.2 in \cite{BI17} for the definition and basic properties of the Biquard connection.
\end{Example}

The following lemmas will be used several times.

\begin{lemma}\label{skew J}
Let $(\M, \Ho,g)$ be a totally geodesic foliation such that $\nabla_\Ho T=0$, i.e. $\nabla_XT=0$ for every $X \in \s\Ho$. Then
\begin{equation}
(\nabla_Z J)_W = -(\nabla_W J)_Z, \qquad \forall \, W,Z \in \s{T\M}.
\end{equation}
\end{lemma}

\begin{proof}
Since the torsion is horizontally parallel, and by definition \eqref{Jmap} of $J$, we only need to prove the statement for $Z,W \in \s\V$. Let $R$ denote the Riemann curvature tensor of $\nabla$. Using the first Bianchi identity, with $\circlearrowright$ denoting the cyclic sum, we have for any $Z \in \Gamma (\V)$ and $X,Y \in \Gamma (\Ho)$
\begin{equation}
0 = \langle R(X,Y) Z, Z \rangle = \langle \circlearrowright R(X,Y)Z, Z  \rangle = \langle (\nabla_Z T)(X,Y) ,Z \rangle.
\end{equation}
Therefore $(\nabla_Z J)_Z=0$ for all $Z \in \s\V$, which implies the statement.
\end{proof}

\begin{lemma} \label{prop:TGCurvature}
Let $(\M, \Ho,g)$ be a totally geodesic foliation with $\nabla_\Ho T = 0$. Let $R$ denote the Riemann curvature tensor of $\nabla$.  Define for $U,V,W \in T\M$,
\begin{equation}
R_\Ho(U,V)W = R(U_\Ho , V_\Ho ) W_\Ho , \qquad R_\V(U,V)W =  R(U_\V , V_\V ) W_\V .
\end{equation}
Then,
 \begin{equation}
 R(U,V) W = R_\Ho(U,V) W + R_\V (U, V) W + (\nabla_{W} T)(U, V).
 \end{equation}
\end{lemma}

\begin{proof}
The result follows from considering each of the possible projections, the first Bianchi identity and formulas relating to the anti-symmetric part of the curvature tensor. Note first that since $\nabla$ preserves $\Ho$ and $\V$, we have $\langle R(\, \cdot \,, \, \cdot \, ) V_\V, V_\Ho \rangle = - \langle V_\V, R(\, \cdot \,, \, \cdot \, ) V_\Ho \rangle = 0$. It follows that
\begin{equation} \label{DecompositionR}
\begin{aligned}
& \langle R(X,Y) V, W \rangle - \langle R_\Ho(X,Y) V, W \rangle - \langle R_\V(X,Y) V, W \rangle \\
& = \langle R(X_\Ho,Y_\V) V_\Ho, W_\Ho \rangle + \langle R(X_\V,Y_\Ho) V_\Ho, W_\Ho \rangle + \langle R(X_\V,Y_\V) V_\Ho, W_\Ho \rangle \\
& \quad  + \langle R(X_\Ho ,Y_\Ho) V_\V, W_\V \rangle + \langle R(X_\Ho ,Y_\V) V_\V, W_\V \rangle + \langle R(X_\V,Y_\Ho) V_\V, W_\V \rangle 
\end{aligned}
\end{equation}
Using the first Bianchi identity, we have
$$\langle R(X_\V, Y_\V) V_\Ho, W_\Ho \rangle = \langle \circlearrowright R(X_\V, Y_\V) V_\Ho, W_\Ho \rangle =0,$$
and we similarly have $\langle R(X_\Ho, Y_\Ho) V_\V, W_{\V} \rangle = \langle (\nabla_{V_\V} T)(X_\Ho, Y_\Ho), W_\V \rangle$.
To obtain the remaining terms of \eqref{DecompositionR}, we will use the following result found in \cite[Appendix]{BaGr17}. Define the tensor
\begin{equation*}
A(X,Y):= T(X,Y)  -  J_X Y -  J_Y X.
\end{equation*}
Then for any connection preserving the metric, we have
\begin{align*}
2\langle R(X,Y) V,  W \rangle - 2\langle R(V,W) X, Y \rangle =&  \langle (\nabla_{X} A)(Y,V) -(\nabla_{Y} A)(X ,V), W \rangle \\ 
&+ \langle (\nabla_{W} A)(V,X) - (\nabla_{V} A)(W,X), Y \rangle.
\end{align*}
If we use the property $(\nabla_X J)_Y = - (\nabla_Y J)_X$, we obtain
\begin{equation} \label{BottAntiSym}
\langle R(X,Y) V,  W \rangle - \langle R(V,W) X, Y \rangle =  \langle (\nabla_{V} T)(X, Y),  W \rangle  -  \langle (\nabla_{X} T)(V, W) ,Y \rangle.
\end{equation}
Using equation \eqref{BottAntiSym}, we have
\begin{align*}
& \langle R(X_\Ho,Y_\V) V_\Ho,  W_\Ho \rangle = \langle R(X_\Ho,Y_\V) V_\Ho,  W_\Ho \rangle - \langle R(V_\Ho,W_\Ho) X_\Ho, Y_\V \rangle \\
& =  \langle (\nabla_{V_\Ho} T)(X_\Ho, Y_\V),  W_\Ho \rangle  -  \langle (\nabla_{X_\Ho} T)(V_\Ho, W_\Ho) ,Y_\V \rangle =0
\end{align*}
and similarly $\langle R(X_\V, Y_\Ho)V_\V, W_\V\rangle =0$.
Inserting all of these identities into \eqref{DecompositionR}, we have the result.
\end{proof}

\subsection{H-type foliations}

\begin{definition}\label{htype}
We say that $(\M,\Ho,g)$ is an H-type foliation if for every $Z \in \Gamma (\V)$ and $X,Y  \in \Gamma (\Ho)$, 
\begin{align}\label{jhgfr}
\langle J_Z X ,J_Z Y \rangle = \| Z\|^2\langle  X , Y \rangle.
\end{align}
Moreover:
\begin{itemize}
\item If the horizontal divergence of the torsion of the Bott connection is zero, then we say that $(\M,\Ho,g)$ is an H-type foliation of Yang-Mills type.
\item If the torsion of the Bott connection is horizontally parallel, i.e.\ $\nabla_\Ho T=0$, then we say that $(\M,\Ho,g)$ is an H-type foliation with horizontally parallel torsion.
\item If the torsion of the Bott  connection is completely parallel, i.e.\ $\nabla T=0$, then we say that $(\M,\Ho,g)$ is an H-type foliation with parallel torsion.
\end{itemize}
\end{definition}

\begin{Remark}\label{normalization}
We note that due to the normalization \eqref{Jmap}, the unit odd-dimensional  sphere $\Sph^{2n+1}$ with its canonical metric is not H-type for the standard Reeb foliation, since one can compute that in that case
\[
\langle J_Z X ,J_Z Y \rangle = 4\| Z\|^2\langle  X , Y \rangle.
\]
However, if one considers the canonical variation of the metric $g$ given by $g_\ve =g_\Ho \oplus \frac{1}{\ve} g_\V$, $\ve >0$, then $(\M,\Ho,g_\ve)$ is a totally geodesic foliation and the corresponding $J$-map is given by $J^\ve=\frac{1}{\ve} J$. Thus, if $(\M,\Ho,g)$ is a totally geodesic foliation such that 
\begin{align}\label{jhgfr2}
\langle J_Z X ,J_Z Y \rangle = \lambda \| Z\|^2\langle  X , Y \rangle,
\end{align}
for some $\lambda >0$, then $(\M,\Ho,g_\lambda)$ is an H-type foliation. This rescaling  does not affect the intrinsic sub-Riemannian geometry of the triple $(\M,\Ho,g_\Ho)$. The condition \eqref{jhgfr2} is a special case of a generalized H-type condition introduced for Carnot groups in \cite[Definition 8]{Rizzi18}.
\end{Remark}

\begin{Remark}
Obviously, parallel torsion $\implies$ horizontally parallel torsion $\implies$ Yang-Mills.  The Yang-Mills assumption plays an important role in the theory of generalized curvature dimension inequalities (see \cite{BG17}, \cite{GrTh16a}) and will be shown to always be satisfied, see Theorem \ref{th:AllYM}. The meaning of the other two assumptions will be apparent in the next sections.
\end{Remark}

\begin{Remark}
As a consequence of the H-type condition, one has for every $X \in \s{\Ho}, Z \in \s{\V}$,  $-\pi_\V ([X,J_ZX])=T(X,J_ZX)= \| X \|^2 Z$. Thus for any $X \in \s\Ho$, we have that $T\M$ is generated by $[X,\Ho]$ and $\Ho$ and in particular $\Ho$ is automatically bracket-generating.
\end{Remark}

\begin{table}
\centering
\scalebox{0.8}{
\begin{tabular}{| l | l | l |}
\hline 
\textbf{Structure} &  Torsion & Reference \\ \hline 
\hline \multicolumn{3}{|l|}{\textbf{Complex Type, \(m = 1, n = 2k\)}} \\ \hline
K-Contact &  YM & \cite{BRcontact} \cite{BW14}\\ \hline
Sasakian  & CP & \cite{BRcontact}  \cite{BoG} \\ \hline
Heisenberg Group   & CP &  \cite{Calin09}  \\ \hline
Hopf Fibration  \(\Sph^1 \hookrightarrow \Sph^{2k+1} \to \CP^k\)  & CP & \cite{BW1} \\ \hline
Anti de-Sitter Fibration  \(\Sph^1 \hookrightarrow \AdS^{2k+1}(\Comp) \to \CH^k\) &  CP &  \cite{Bad02} \cite{W1} \\ \hline
\hline \multicolumn{3}{|l|}{\textbf{Twistor Type, \(m = 2, n = 4k\)}} \\ \hline
Twistor space over quaternionic K\"ahler manifold & HP & \cite{Hadfield14} \cite{Sal}  \\ \hline
Projective Twistor space   \(\CP^1 \hookrightarrow \CP^{2k+1} \to \Quat P^k\) & HP & \cite{BW2} \\ \hline
Hyperbolic Twistor space   \(\CP^1 \hookrightarrow \CH^{2k+1} \to \Quat H^k\) & HP & \cite{BDW18} \cite{Bad02}  \\ \hline
\hline \multicolumn{3}{|l|}{\textbf{Quaternionic Type, \(m = 3, n = 4k\)}} \\ \hline
3K-contact  & YM & \cite{Je} \cite{Tanno2} \\ \hline
Negative 3K-contact   & YM & \cite{Je} \cite{Tanno2} \\ \hline
3-Sasakian  & HP & \cite{MR1798609} \cite{RS-3sas} \\ \hline
Negative 3-Sasakian   & HP &  \cite{MR1798609} \\ \hline
Torus bundle over hyperk\"ahler manifolds & CP & \cite{Her} \\ \hline
Quaternionic Heisenberg Group & CP &\cite{Calin09}  \\ \hline
Quaternionic Hopf Fibration  \(\SU(2) \hookrightarrow \Sph^{4k+3} \to \Quat P^k\)  & HP & \cite{BW2} \\ \hline
Quaternionic Anti de-Sitter Fibration  \(\SU(2) \hookrightarrow \AdS^{4k+3}(\Quat) \to \Quat H^k\)  & HP &  \cite{BDW18} \cite{Bad02}  \\ \hline
\hline \multicolumn{3}{|l|}{\textbf{Octonionic Type, \(m = 7, n = 8\)}} \\ \hline
Octonionic Heisenberg Group & CP & \cite{Calin09}  \\ \hline
Octonionic Hopf Fibration  \(\Sph^7 \hookrightarrow \Sph^{15} \to \OP^1\) & HP & \cite{Ornea13} \\ \hline
Octonionic Anti de-Sitter Fibration  \(\Sph^7 \hookrightarrow \AdS^{15}(\Octo) \to \OH^1\) & HP & \cite{Bad02}  \\ \hline
\hline \textbf{H-type Groups}, $m$ is arbitrary  & CP & \cite{Cowling91}  \cite{Kaplan} \\ \hline
\end{tabular}}
\caption{Some examples of H-type foliations.}
\label{Table 3}

\end{table}

From the H-type condition, $\Ho_p$ is for every $p \in \M$ a $\Cl(\V_p)$-module, where $\Cl(\V_p)$ denotes the Clifford algebra of $\V_p$.  Algebraic properties of Clifford modules are  well-known a (see for instance \cite{Cowling91}) and we shall make  use of  some of the most basic ones without further reference. To motivate the study of H-type foliations and stress that they provide a unified framework for many structures previously studied in the literature, we  point out in Table \ref{Table 3} several distinguished classes. More examples will be obtained as a consequence of the results of  Section \ref{s:curvature constancy} (see the two tables in the Introduction). For the Torsion column, YM means Yang-Mills, HP means horizontally parallel and CP means completely parallel. As a possible guide to the reader, in the Reference column we point out some references in the literature where the structure has been studied, sometimes with a sub-Riemannian point of view.

\subsection{Quaternionic structures}\label{quaternionic structures}

In this section, we introduce a  remarkable subclass of H-type foliations which encompass 3K-contact and negative 3K-contact manifolds. Let $(\M,\Ho,g)$ be an H-type foliation.
Consider the map $Z \to J_Z$. By the universal property of Clifford algebras,  at any $p \in \M$, such map can  uniquely be extended into a bundle algebra homomorphism, still denoted $J$, from the Clifford algebra $\Cl (\V_p)$ to the algebra of horizontal endomorphisms $\End (\Ho_p)$, where the product on $\End (\Ho_p)$ is given by the composition rule of operators. The Clifford multiplication will be denoted by a dot $\cdot$.  In particular, we explicitly note that $J_{1} = \Id_\Ho$ and $J_{v\cdot w} = J_v J_w$.

\begin{lemma}\label{algebra}
Let $(\M,\Ho,g)$ be an H-type foliation. Let $p \in \M$.
\begin{itemize}
\item Consider $\End(\Ho_p)$ as a Lie algebra with Lie brackets being the usual commutator brackets. Define $\fa(p)$ as the Lie subalgebra  generated by maps $J_z$, $z \in \V_p$. Then one of the following holds:
\begin{enumerate}[\rm (i)]
\item $\fa(p) = \{ J_z \, : \, z \in \V_p \}$ and it is isomorphic to either $\R$ or $\fso(3)$;
\item $\fa(p) = \{ J_{z_0}, [J_{z_0}, J_{z_1}] \, : \, z_0, z_1 \in \V_p \}$ and it is isomorphic to $\fso(m+1)$.
\end{enumerate}
\item Assume that  $\{ J_z \, : z \in \V_p \}$ forms a Lie algebra under the commutator brackets. Define
$$\fA(p) = \{  J_z \, : \, z  \in \R  \oplus \V_p \}.$$
Then $\fA(p)$, with product given by the composition of endomorphisms is a field  isomorphic to the field of complex numbers $\Comp$ or the field of quaternions $\Quat$.
\end{itemize}
\end{lemma}

\begin{proof}

If $m = 1$, then for every $p \in \M$, $\fa(p)$ is a Lie algebra isomorphic to $\R$ and $\fA(p)$ a field isomorphic to $\Comp$, so we assume that $m \ge 2$.  If we endow the vector space $\V_p \oplus \Cl_2 (\V_p )$ with the Lie bracket $Z_1\cdot Z_2 -Z_2 \cdot Z_1$,
then $\V_p \oplus \Cl_2 (\V_p )$ is a Lie algebra isomorphic to $\fso(m+1)$. The map $Z \mapsto J_Z$ is a surjective Lie algebra homomorphism between $\V_p \oplus \Cl_2 (\V_p )$ and  the Lie subalgebra $\fa(p)$ of $\End( \Ho_p)$. If $m\neq 3$, then the Lie algebra $\fso(m+1)$ is simple, so we actually have a Lie algebra isomorphism. Therefore $\fa(p)$ is isomorphic to $\fso(m+1)$. If $m = 3$, then $\fso(4)$ is isomorphic to $\fso(3) \oplus \fso(3)$. So the surjective Lie algebra homomorphism   $\V_p \oplus \Cl_2 (\V_p )\to \fa(p)$ is either an isomorphism, in which case, we conclude as for $m \neq 3$, or $\fa(p)$ is isomorphic to $\fso(3)$. 

We now prove the second part of the statement. If the maps $J_z$ form a Lie algebra and $m\ge 2$, then from the previous argument, we have $m=3$. Let then $z_1,z_2 \in \V_p$ be such that $\| z_1\|=\|z_2\|=1$ and $z_1 \perp z_2$. Denote $z_3$ as the element such that
$
\frac{1}{2} [J_{z_1}, J_{z_2} ] = J_{z_3}.
$
It is easily seen that due to the properties of $J$, the triple $z_1,z_2,z_3$ is an orthonormal basis for $\V_p$ such that $J_{z_1}^2=J_{z_2}^2=J_{z_3}^2=J_{z_1}J_{z_2}J_{z_3}=-\Id_\Ho$. Thus $\fA(p)$ is isomorphic to $\Quat$.
\end{proof}

\begin{lemma}\label{isomorphic algebra}
Let $(\M,\Ho,g)$ be an H-type foliation with horizontally parallel torsion. Then for any $p,q \in \M$, $\fa(p) $ is isomorphic to $\fa(q) $.
\end{lemma}

\begin{proof}
Let $\gamma$ be an horizontal curve  (i.e. $\gamma' \in \Ho_\gamma$) joining $p$ to $q$. Such a curve always exists from the Chow-Rashevskii theorem since $\Ho$ is bracket-generating from the H-type condition. The $\nabla$-parallel transport along $\gamma$ induces a Lie algebra isomorphism between $\fa(p) $ and $\fa(q) $, because $\nabla_\Ho J=0$. Let $P_{\gamma} : T_p \M \to T_q \M$ be the $\nabla$-parallel transport along $\gamma$ as above. Thanks to the properties of $\nabla$, $P_{\gamma}$ maps $\Ho_p$ to $\Ho_q$ and $\V_q$ to $\V_q$. Then it induces a Lie algebra isomorphism $\Theta: \End(\Ho_p) \to \End(\Ho_q)$ given by
\[
\Theta(A) = P_\gamma \circ A \circ P_{\gamma}^{-1}, \qquad \forall A\in \End(\Ho_p).
\]
Furthermore, if $Z\in \s\V$ and $X\in \s\Ho$ are parallel along $\gamma$, then also $J_Z X$ is parallel along $\gamma$ since $\nabla_\Ho J=0$. It follows that for any $z\in \V_p$ we have $\Theta(J_z) = J_{P_{\gamma}z}$, so that $\Theta$ maps generators of $\fa(p)$ to generators of $\fa(q)$, concluding the proof.
\end{proof}

\begin{definition}\label{quaternionic}
Let $(\M,\Ho,g)$ be an H-type foliation.  We say that $(\M,\Ho,g)$ is a quaternionic type foliation if for  every $p \in \M$, $\fA(p) \simeq \Quat$.
\end{definition}
In particular, the leaves of the foliations have dimension $m =3$. We also note from Lemma \ref{isomorphic algebra}, that if $(\M,\Ho,g)$ is an H-type foliation with horizontally parallel torsion, then for it to be quaternionic, it is enough that $\fA(p) \simeq \Quat$ at some point $p \in \M$.

\begin{Example}
The examples  given in Table \ref{Table 3} under the category Quaternionic Type are examples of such structures.
\end{Example}

\begin{Remark}
While the quaternions yield rich classes of structures, octonions do not. Indeed, we will see that for $m=7$, $n=8$, the octonionic Heisenberg group, the octonionic Hopf fibration and the octonionic anti de-Sitter fibration are the only examples of simply connected H-type submersions that carry a parallel horizontal Clifford structure, see Section 3.3 and Tables \ref{Table 1}, \ref{Table 2}. However, those three examples are still algebraically remarkable, because even though $J_{\R  \oplus \V} $ is not an algebra, one has for every $X \in \s{\Ho}$, $J_{\R  \oplus \V} J_{\R  \oplus \V} X= J_{\R  \oplus \V} X$. This is the so-called $J^2$ condition in Clifford modules, see \cite{Calin09} and \cite{Cowling91}.
\end{Remark}

\subsection{H-type foliations are Yang-Mills}

Although H-type foliations are not necessarily horizontally parallel, they are always Yang-Mills.  The importance of this result will be shown in Section \ref{section CD}.

\begin{theorem} \label{th:AllYM}
Let $(\M , \Ho, g)$ be an H-type foliation. Then it satisfies the Yang-Mills condition.
\end{theorem}
To prove this result, we will use the following Lemma.
\begin{lemma} We have the following relations for the covariant derivatives of~$T$ and~$J$.
\begin{enumerate}[\rm(a)]
\item If $X,Y,Z \in \s\Ho$, then $\circlearrowright (\nabla_X T)(Y,Z) =0$.
\end{enumerate}
Furthermore, for arbitrary vector fields $X,Y\in \s{T\M}$ and a vertical vector field $W\in \s\V$, the following relations hold.
\begin{enumerate}[\rm (a)]
\setcounter{enumi}{1}
\item $J_W (\nabla_X J)_W Y= - (\nabla_X J)_W J_WY$,
\item $\langle X, (\nabla_Y J)_W X \rangle =0$,
\item $\langle J_W X, (\nabla_Y J)_W X \rangle =0$,
\item $(\nabla_{J_W X} J)_W X = (\nabla_{X} J)_W J_W X = - J_W (\nabla_{X} J)_W X$. In particular, $$(\nabla_{J_W X} J)_W J_W X = - \| W\|^2 (\nabla_X J)_W X.$$
\end{enumerate}
\end{lemma}
\begin{proof}
The first relation is a result of the Bianchi identity. For any vertical vector field $W \in \s\V$, we have $\langle \circlearrowright (\nabla_X T)(Y,Z), W \rangle = \langle \circlearrowright R(X,Y) Z , W \rangle = 0$. Taking the covariant derivative of the H-type condition
\begin{equation}
J_W^2 = - \|W\|^2 \Id_\Ho,
\end{equation}
proves (b). Property (c) follows from the skew-symmetry of $(\nabla_YJ)_W$. Property (d) follows from (b). Finally, for Property (e), we note that
\begin{align}
0 & = \langle \circlearrowright (\nabla_{J_WX } T)(X, Y), W \rangle \\
& = \langle (\nabla_{J_WX } T)(X, Y) - (\nabla_X T)(J_W X, Y), W \rangle + \langle (\nabla_{Y } T)(J_WX , X), W \rangle \\
& = \langle (\nabla_{J_WX } J)_W X  - (\nabla_X J)_W J_W X, Y \rangle - \langle  (\nabla_{Y } J)_W X , J_WX  \rangle \\
& = \langle (\nabla_{J_WX } J)_W X  - (\nabla_X J)_W J_W X, Y \rangle,
\end{align}
which completes the proof, as $Y$ was arbitrary. We note also from (b) that
\begin{equation*}
(\nabla_{J_W X} J)_W J_W X = - J_W (\nabla_{J_W X} J)_W X = J_W^2 (\nabla_X J)_W X = - \| W \|^2 (\nabla_X J)_W X.\qedhere
\end{equation*}
\end{proof}

\begin{proof}[Proof of Theorem~\ref{th:AllYM}]
Let $p \in \M$ be arbitrary. Note that if $Z\in \V_p$ is a unit vector and if $X_1, \dots, X_n$ is an orthonormal basis of $\Ho_p$, then so is $J_{Z} X_1, \dots, J_{Z} X_n$. Hence for any horizontal $Y$, we have
\begin{align*}
 \langle \tr_{\Ho} (\nabla_\times T)(\times, Y), Z \rangle = \sum_{i=1}^n \langle (\nabla_{J_Z X_i}J)_Z J_Z X_i , Y \rangle  = - \sum_{i=1}^n \langle (\nabla_{X_i}J)_Z X_i , Y \rangle = - \langle \tr_{\Ho} (\nabla_\times T)(\times, Y), Z \rangle.
\end{align*}
Hence $\tr_\Ho (\nabla_\times T)(\times, \, \cdot \,) = 0$ and the foliation is Yang-Mills.
\end{proof}

\subsection{Curvature dimension inequalities on H-type foliations}\label{section CD}

In this  subsection we show that on H-type foliations, the generalized curvature dimension condition introduced in \cite{BG17} is only controlled by the horizontal Ricci curvature and deduce several corollaries. Let $(\M,\Ho,g)$ be an H-type foliation. We assume that the metric $g$ is complete.
The Riemannian gradient will be denoted $\nabla$ and we write the horizontal gradient as $\nabla_\Ho$, which is the projection of~$\nabla$ onto~$\Ho$.  Likewise, $\nabla_\V$ will denote the vertical gradient.  Let $\mu_g$ denote the Riemannian volume measure. The  horizontal Laplacian $\Delta_\Ho$ of the foliation is the generator of the symmetric closable bilinear form in $L^2(\M,\mu_g)$:
\begin{equation}
\mathcal{E}_{\Ho} (u,v) 
=\int_\M \langle \nabla_\Ho u , \nabla_\Ho v \rangle\,d\mu_g, 
\quad u,v \in C_0^\infty(\M).
\end{equation}

We adopt the convention that $\Dh$ is a negative operator. The H-type hypothesis implies  that $\Ho$ is bracket-generating, therefore it follows from H\"ormander's theorem that the horizontal Laplacian $\Delta_{\Ho}$ is locally subelliptic. The completeness assumption on the Riemannian metric $g$ implies that $\Delta_{\Ho}$ is essentially self-adjoint on the space of smooth and compactly supported functions (see for instance \cite{Strichartz} or Proposition 5.1 in  \cite{BaudoinEMS2014}).

\begin{Remark}
For H-type foliations, one can easily check that the Riemannian measure $\mu_g$ is proportional to the intrinsic Popp's measure of the sub-Riemannian structure obtained by the restriction $g|_{\Ho}$. Therefore, the operator $\Delta_{\Ho}$ defined above coincides with the intrinsic sub-Laplacian (see \cite{BR-Popp} and \cite[Section 10.6]{Montgomery}). As such, we will indifferently refer to $\Dh$ as the horizontal Laplacian or the sub-Laplacian.
\end{Remark}

We denote by $\ri_\Ho$ the horizontal Ricci curvature of $(\M, \Ho, g)$ i.e. the horizontal trace of the Riemann curvature tensor of the Bott connection.

\begin{proposition}\label{generalized CD}
Let $(\M, \Ho, g)$ be an  H-type foliation  such that $\ri_\Ho \ge K g_\Ho$ with $K \in \R$. Then $(\M, \Ho, g)$ satisfies the generalized curvature dimension inequality CD$\left( K, \frac{n}{4}, m , n \right)$, i.e.\ for every $f \in C^\infty(\M)$ and $\ve >0$, one has the following Bochner's type inequality:
\begin{multline*}
 \frac{1}{2} \left( \Dh \| \nabla_{\Ho} f\|^2 -2 \langle \nabla_{\Ho} f, \nabla_{\Ho} \Dh  f \rangle \right)+ \frac{\ve}{2} \left( \Dh \| \nabla_{\V} f\|^2 -2 \langle \nabla_{\V} f, \nabla_{\V} \Dh  f \rangle \right) \\
 \ge   \frac{1}{n} (\Dh f)^2 +\left( K-\frac{m}{\ve}\right) \| \nabla_{\Ho} f\|^2+\frac{n}{4} \| \nabla_{\V} f\|^2.
\end{multline*}
\end{proposition}

\begin{proof}
The key point is that H-type foliations are Yang-Mills (see Theorem \ref{th:AllYM}). The proof is then similar to the proof of this result in the case of Sasakian foliations (see Theorem 2.24 in \cite{BG17}), so we omit it for conciseness, but refer to Remark 2.25 in \cite{BG17}.
\end{proof}

As a corollary  from Proposition \ref{generalized CD} and \cite{BaudoinEMS2014, BBG14, BG17} one deduces the following results:

\begin{corollary} \label{yang-mills corollary}
Let $(\M, \Ho, g)$ be a complete  H-type foliation  with $\ri_\Ho \ge K g_\Ho$ with $K \in \R$. Let us denote by  $d$ the sub-Riemannian (a.k.a.\ Carnot-Carath\'eodory) distance.

\begin{enumerate}
\item If $K \ge 0$, then the metric measure space  $(\M, d, \mu)$ satisfies the volume doubling property and  supports a 2-Poincar\'e inequality, i.e. there exist constants $C_D, C_P>0$, depending only on $K,n,m$, for which one has for every $p \in \M$ and every $r>0$:
\begin{equation}\label{dcsr}
\mu(B(p,2r)) \le C_D\ \mu(B(p,r)),
\end{equation}
\begin{equation}\label{pisr}
\int_{B(p,r)} |f - f_B|^2 d\mu_g \le C_P r^2 \int_{B(p,r)} \| \nabla_{\Ho} f \|^2 d\mu_g,
\end{equation}
for every $f\in C^1( B(p,r))$, where we have let $f_B = \mu_g(B)^{-1} \int_B f d\mu_g$, with $B = B(p,r)$.

\item If $K >0$, then $\M $ is compact with a finite fundamental group and
\begin{align*}
\diam (\M,d)\le  2 \sqrt{3} \pi \sqrt{ \frac{(n+4m)(n+6m)}{nK} }.
 \end{align*}
\item If $K >0$, then the first non zero eigenvalue of the sub-Laplacian $-\Delta_{\Ho}$ satisfies 
\begin{equation}
\lambda_1 \ge \frac{nK}{ n+3m-1}.
\end{equation}
\end{enumerate}
\end{corollary}
\begin{proof}
Point 1. follows from \cite[Theorem 1.5]{BBG14}, and Point 2. from \cite[Theorem 10.1]{BG17} or \cite[Theorem 6.1]{BaudoinEMS2014} for a simpler proof. Point 3 follows from \cite[Theorem 4.9]{BaudoinEMS2014} with the values, $\rho_1=K$, $\rho_2=n$, $\kappa=m$ and $d=n$. 
\end{proof}

\begin{Remark}
The volume doubling property and 2-Poincar\'e inequality are central for the validity of covering theorems of Vitali-Wiener type, maximal function estimates, and represent the central ingredients in the development of analysis and geometry on metric measure spaces,  see for instance \cite{Heinonen01} and the more recent \cite{Shanmu15}. It is not known if the generalized curvature dimension implies the significantly stronger $1$-Poincar\'e inequality. We point out that the diameter upper bound which is obtained when $K>0$ is not sharp. In the subsequent paper \cite{BGMR19},  under stronger geometric assumptions (lower bounds on partial traces of the tensor $R_\Ho$), both the 1-Poincar\'e inequality (actually even the measure contraction property) and sharp diameter upper bounds are proved.
\end{Remark}

More consequences of the generalized curvature dimension inequality are given in \cite{BaudoinEMS2014, BBG14, BG17},  for instance  Li-Yau estimates for non negative  solutions of the sub-Riemannian heat equation or subelliptic Sobolev and log-Sobolev inequalities.

\section{Horizontal Clifford structures}

We now turn to the second part of the paper and study  H-type foliations that carry a parallel horizontal Clifford structure.  One should have the understanding that H-type foliations with a parallel horizontal Clifford structure are to general H-type foliations what Sasakian and 3-Sasakian manifolds are respectively to K-contact and 3K-contact manifolds.

\subsection{Parallel horizontal Clifford structures}

\begin{definition}\label{def parallel Clifford}
Let $(\M,\Ho,g)$ be an H-type foliation with horizontally parallel torsion. We say that  $(\M,\Ho,g)$  is an H-type foliation with a parallel horizontal Clifford structure if there exists 
a smooth bundle map $\Psi: \V \times \V \to \Cl_2 (\V )$ such that for every $Z_1,Z_2 \in \s{\V}$
\begin{align}\label{def Psi}
(\nabla_{Z_1} J)_{Z_2}=J_{\Psi(Z_1,Z_2)}.
\end{align}
\end{definition}

\begin{Remark}
If $m=1$, then the parallel horizontal Clifford assumption is always satisfied with $\Psi=0$. 
\end{Remark}

\begin{proposition}\label{uniqueness psi}
Let $(\M,\Ho,g)$ be a H-type foliation with parallel horizontal Clifford structure. Then the map $\Psi$ is unique.
\end{proposition}
\begin{proof}
The proposition follows from the following fact: at any $p \in \M$, the map $J : \Cl_2 (\V_p )   \to  \End( \Ho_p)$ defined by the restriction of $Z\mapsto J_Z$ to $\Cl_2 (\V_p )$ is injective. We prove this claim.
If $m$ is even, then $\Cl (\V_p )$ is a central simple algebra, thus the map $J:  \Cl (\V_p )   \to  \End( \Ho_p)$ is injective and so is the restriction $J : \Cl_2 (\V_p )   \to  \End( \Ho_p)$. If $m$ is odd, then the even Clifford algebra  $\Cl^0 (\V_p )$ is central simple. Thus the map $J:  \Cl^0 (\V_p )   \to  \End( \Ho_p)$ is injective and so is the restriction $J : \Cl_2 (\V_p )   \to  \End( \Ho_p)$.
\end{proof}

We have the following lemma concerning some algebraic properties of the map $\Psi$.

\begin{lemma}\label{property Psi}
Let $\Psi$ be defined by \eqref{def Psi}. Then, for every $u,v \in \V$ we have
\begin{enumerate}
\item $\Psi(u,v)=-\Psi(v,u)$;
\item $\Psi(u,v)\cdot v+v\cdot \Psi(u,v)=0$.
\end{enumerate}
\end{lemma}
\begin{proof}
Fix non zero $u,v \in \V_p$. The first statement follows from $(\nabla_u J)_v=-(\nabla_v J)_u$ and the uniqueness of $\Psi$. For the second one, since $\Psi(u,v) \in  \Cl_2 (\V_p )$, one can find $a,b \in \Cl_2 (\V_p )$ such that $\Psi(u,v) = a + b$, and such that $v \cdot a = -a\cdot v$ and $v \cdot b = b \cdot v$. The second statement is then equivalent to $b=0$. If we apply $\nabla_u$ to the relation $J_v J_v=- \langle v,v\rangle  \Id_\Ho$ one obtains that $\Psi(u,v)\cdot v+v\cdot \Psi(u,v)$ belongs to the kernel of $J : \Cl (\V_p )   \to  \End(\Ho_p)$.  Therefore, we obtain
\begin{equation}
(a+ b)\cdot v + v \cdot (a+b) \in \ker J.
\end{equation}
Using the properties of $a,b$ we obtain $b\cdot v \in \ker J$ or, equivalently, $b \in \ker J$. Since $b \in  \Cl_2 (\V_p )$ and  $J : \Cl_2 (\V_p )   \to  \End( \Ho_p)$ is injective by the proof of Proposition~\ref{uniqueness psi}, we have $b=0$.
\end{proof}

The previous lemma imposes strong algebraic conditions on $\Psi$. The next theorem characterizes the set of possible expressions $\Psi$ may have. We shall first need the following lemma that allows us to relate the norm of $\nabla J$ to the sectional curvature of the leaves of the foliation.

\begin{lemma}\label{vertical sectional curvature}
Let \((\M,\Ho,g)\) be a H-type foliation with horizontally parallel torsion. Then for every \(X \in \s\Ho, Z,W \in \s\V\),
\begin{equation}
\| (\nabla_{Z} J)_{W} X \|^2=\langle R(Z,W)W, Z \rangle \| X \|^2. 
\end{equation}
\end{lemma}

\begin{proof}
Recall first the second Bianchi identity for connections with torsion,
\begin{equation}
\circlearrowright (\nabla_u R)(v,w) + \circlearrowright R(T(u,v),w) = 0.
\end{equation}
 From Lemma \ref{prop:TGCurvature} we have that \(R(X,Y)W = (\nabla_{W}T)(X,Y)\) and \(R(Z,X) = 0\) for any \(X,Y \in \s\Ho, Z,W \in \s\V\) and so
\begin{equation}
- R(T(X,Y),Z)W = (\circlearrowright (\nabla_{X} R)(Y,Z))W = (\nabla_{Z,W}^2 T)(X,Y). 
\end{equation}
We now note that for any $W \in \s\V$ and $X \in \s\Ho$, we have as a consequence of the H-type condition
\begin{equation}
T(X, J_W X) = \| X\|^2 W.
\end{equation}
Therefore, we have
\begin{equation}\label{eq:seconder}
\langle R(Z,W)W, Z \rangle {\| X\|^2} = -\langle (\nabla_{Z,Z}^2 T)(X, J_{W}X) , W \rangle =- \langle (\nabla_{Z,Z}^2 J)_{W} X, J_{W}X \rangle.
\end{equation}
We then claim that
\begin{equation}
\|(\nabla_{Z}J)_{W} X\|^2 = \langle R(Z,W)W, Z \rangle  \| X\|^2.
\end{equation}
Since both sides are tensors, it is sufficient to prove the above identity at any given $p \in \M$. Therefore, once we have fixed $p$, we can assume, without loss of generality, that $\nabla_Z Z=\nabla_Z W = \nabla_Z X = 0$ along the geodesic with initial vector $Z(p)$. In this case the following holds:
\begin{align}
0 & = \tfrac{1}{2}\nabla^2_{Z,Z}\langle J_W X,J_W X\rangle \\
& = \nabla_Z \langle (\nabla_Z J)_W X,J_W X\rangle \\
& =  \|  (\nabla_Z J)_W X \|^2 + \langle (\nabla^2_{Z,Z} J)_W X, J_W X\rangle\\
& =   \|  (\nabla_Z J)_W X \|^2 -\langle R(Z,W)W, Z \rangle \| X\|^2,
\end{align}
where everything is computed at $p$, and in the last line we used \eqref{eq:seconder}.
\end{proof}

 We are now in position to prove the following theorem.

\begin{theorem}\label{Formula psi}
Let $(\M,\Ho,g)$ be an H-type foliation with parallel horizontal Clifford structure. Then there exists a constant $\kappa \in \R$ such that for every $u,v \in \V_p$, $p \in \M$
\begin{equation}\label{eq:psi}
\Psi (u,v) =-\kappa ( u \cdot v +\langle u, v \rangle),
\end{equation}
where $u \cdot v$ denotes the product in the Clifford algebra $\Cl(\V_p)$. Moreover the sectional curvature of the leaves of the foliation associated to $\V$ is constantly equal to~$\kappa^2$. In particular, if the torsion is completely parallel, the leaves are flat.
\end{theorem}

\begin{proof}
We first remark that by linearity, and since $\Psi$ is skew-symmetric and takes values in $\Cl_2(\V)$, it is sufficient to prove \eqref{eq:psi} for unit vectors satisfying $u \perp v$. In this case, fix an orthonormal basis for $\V$ given by $u,v, w_1,\dots,w_{m-2}$. Since $\Psi$ takes values on $\Cl_2(\V)$, we have
\begin{equation}\label{eq:firsteq}
\Psi(u,v) = \psi_{uv} u \cdot v + \sum_{i=1}^{m-2} \psi_{ui} u \cdot w_i + \sum_{i=1}^{m-2} \psi_{vi} v \cdot w_i + \sum_{i< j} \psi_{ij} w_i\cdot w_j,
\end{equation}
for some $\psi_{uv},\psi_{ui},\psi_{vi} \in \R$. Using Lemma~\ref{property Psi} we obtain $\psi_{ui}=\psi_{vi}=\psi_{ij}=0$. Using again Lemma~\ref{property Psi} combined with the bilinearity of $\Psi$ one also obtains that $\psi_{uv} =-\kappa$ does not depend on $u,v$, but may still depend on $p$. Applying Lemma \ref{vertical sectional curvature} with orthonormal \(z,w \in \V_p\) and unit \(u \in \Ho_p\), we obtain for the sectional curvature of the vertical plane generated by $z$ and $w$:
\begin{equation}
\langle R(z,w) w,z \rangle = \|(\nabla_{z} J)_{w} u\|^2 = \|J_{\Psi(z,w)}u\|^2 =  \kappa^2 \|J_{z\cdot w} u\|^2 = \kappa^2.
\end{equation}
We now prove that $\kappa$ is constant as a function on $\M$. For \(X,Y \in \s\Ho\) and orthonormal \(Z,W \in \s\V\), using Lemma \ref{prop:TGCurvature}, we obtain
\begin{equation}
\langle R(X,Y) Z,W \rangle = \langle (\nabla_{Z}J)_{W} X, Y\rangle =- \kappa \langle J_{Z\cdot W} X,Y \rangle.
\end{equation}
Differentiating the above equation with respect to \(V \in \s\Ho\), and summing cyclically over \(V,X,Y\), Bianchi's second identity and the fact that \(\nabla\) is metric imply that
\begin{equation}
0 = \langle \circlearrowright (\nabla_V R)(X,Y) Z ,W \rangle =- \circlearrowright (V \kappa) \langle J_{Z \cdot W} X,Y \rangle.
\end{equation}
By choosing $X = J_W V$ and $Y = J_Z V$, one obtains \(V \kappa = 0\) for all \(V \in \s\Ho\). But this means that $\kappa$ is constant along any curve tangent to $\Ho$, and since \(\Ho\) is bracket-generating, implying that any pair of points can be connected by a horizontal curve, $\kappa$ has to be constant by Remark~\ref{re:Constant}.
\end{proof}

\begin{Remark}
We write \eqref{eq:psi} with $-\kappa$ instead of $\kappa$ because in next sections, we will see that the sign of $\kappa$ is   important and decides of the topology of $\M$ in a crucial way. In particular, we will prove that if $\kappa >0$, then $\M$ is necessarily compact with a finite fundamental group.  See Corollary~\ref{compact finite fundamental corollary}.
\end{Remark}

\subsection{H-type foliations with completely parallel torsion}

We first  study  H-type foliations with completely parallel torsion. This corresponds to a parallel horizontal Clifford structure for which $\kappa= 0$ and so $\Psi=0$. We have the following result that essentially shows that H-type sub-Riemannian manifolds with completely parallel torsion  which are not H-type groups may only exist when $m=1,2$ or $3$.

\begin{theorem}\label{flat leaves}
Let $\pi: (\M,g) \to (\B,h)$ be a Riemannian submersion with totally geodesic fibers. Assume that $\B$ is simply connected and that  $(\M, \Ho, g)$ is an H-type foliation with  completely parallel torsion, where $\Ho$ is the horizontal space of $\pi$. Then one of the following (non exclusive) cases  occur:

\begin{itemize}
\item  $m=1$ and $\B$ is K\"ahler;
\item $m=2$ or $m=3$ and $\B$ is locally hyper-K\"ahler;
\item $m$ is arbitrary and $\B$ is flat, thus isometric to a representation of the Clifford algebra $\Cl (\R^m)$.
\end{itemize}
\end{theorem}

\begin{proof}
From Theorem \ref{Formula psi}, we first note that the fibers of $\pi$ have zero sectional curvature. Let $Z_1,\cdots ,Z_m$ be a local orthonormal vertical frame with $\nabla_{Z_i} Z_j =0$. Since for every $Z \in \s{\V}$, $\nabla_Z (J_{Z_i})=0$, one deduces that $J_{Z_i}$ is projectable onto $\B$. Thus, there exist  $(1,1)$ tensors ~$\bar{J}_{Z_i}$ on~$\B$ such that for any basic vector field $X$ on $\M$, $\bar{J}_{Z_i} \bar{X}= \overline{J_{Z_i} X }$ where $\bar{X}$ denotes the projection of $X$ onto $\B$. Since the Bott connection projects onto the Levi-Civita connection, one deduces that the $\bar{J}_{Z_i}$ are parallel almost complex structures on $\B$. Therefore, if $m=1$ then $\B$ is K\"ahler and if $m \ge 2$, then $\B$ is locally hyper-K\"ahler. Let us now assume that $m \ge 4$. We want to show that $\B$ is flat. The argument is similar to \cite{MS},  proof of Theorem 2.9. We reproduce it in our setting for convenience of the reader. Since $\B$ is locally hyper-K\"ahler, it has to be Ricci flat. Let us first assume that $\B$ is irreducible. Then, from Berger-Simons classification theorem (see \cite{Besse}, page 300), $\B$ is either locally symmetric or its holonomy is included in $\SU(n/2)$, $\Sp (n/4)$ or $\Spin(7)$. If $\B$ is locally symmetric then it is flat due to the fact it is Ricci flat. On the other hand, it is impossible that the holonomy of $\B$ is included in  $\SU(n/2)$, $\Sp (n/4)$ or $\Spin(7)$ because $m \ge 4$ implies that the space of parallel two-forms on $\B$ has dimension at least 4 which is larger than the dimension of the centralizer of the Lie algebras of  $\SU(n/2)$, $\Sp (n/4)$ and $\Spin(7)$. One concludes that $\B$ is flat. If $\B$ is not irreducible, one can use the de Rham decomposition theorem to conclude as above.
\end{proof}

We note that the first case in the previous theorem corresponds to the case where $(\M, \Ho, g)$ is a Sasakian foliation and the last case  corresponds to H-type groups. The second case, when $m=3$ and $(\M, \Ho, g)$ is of quaternionic type corresponds to the hyper $f$-structures considered in \cite{Her}.

Since totally geodesic foliations with bundle-like metric are always locally described by a totally geodesic Riemannian submersion, one deduces the following corollary.

\begin{corollary}
Let  $(\M, \Ho, g)$ be an H-type foliation with  completely parallel torsion. If $m \ge 2$, then $\M$ is horizontally Ricci flat, i.e. $\ri_\Ho =0$ where $\ri_\Ho$ is the horizontal Ricci curvature of the Bott connection. If $m \ge 4$, then $\M$ is horizontally  flat, i.e. $R_\Ho=0$ where $R_\Ho$ is defined as in Lemma \ref{prop:TGCurvature}.
\end{corollary}

\subsection{Parallel horizontal Clifford structures and curvature constancy}\label{s:curvature constancy}

In this section, we show how  H-type foliations with a parallel horizontal Clifford structure  can be obtained from totally geodesic  Riemannian or semi-Riemannian foliations associated with curvature constancy.  Conversely,  all H-type foliations with a parallel Clifford structure arise in this way, up to rescaling the metric in the vertical direction (which does not change the intrinsic geometry of the corresponding sub-Riemannian structure $g_\Ho$). Using  a result from \cite{MS}, this will yield a classification of simply connected  H-type foliations with a parallel horizontal Clifford structure coming from a Riemannian submersion. Let $(\M,g)$ be a semi-Riemannian manifold. Denote by $R^g$ its Riemannian curvature tensor (for the Levi-Civita connection). Following \cite{Gray}, we give the following definition.
\begin{definition} For $\rho \in \R$, the $\rho$-\emph{curvature constancy} of $(\M,g)$ is the distribution given by
\begin{equation}
\calC_p (\rho, g)=\left\{ v \in T_p\M \mid R^g(v,x)y = \rho \left(\langle x,y \rangle_g v -\langle v,y\rangle_g x \right)\quad \forall\, x,y \in T_p \M \right\}, \quad \forall\, p \in \M.
\end{equation}
\end{definition}
As proved in \cite{Gray}, assuming that $ \rank \ \calC_p (\rho, g)$ is constant and $\geq 1$, the $\rho$-curvature constancy is an integrable distribution and the leaves of the corresponding foliation are totally geodesic. If we further assume that the metric is bundle-like along $\calC_p (\rho, g)$, letting $\Ho = \calC_p (\rho, g)^\perp$, we have that $(\M,\Ho,g)$ is a totally geodesic foliation in the sense of Section \ref{s:setting}.  We have then the following theorem:

\begin{theorem}\label{t:curvature constancy1}
Let $(\M,\Ho,g)$ be a totally geodesic foliation with vertical distribution $\V$.  Let $\kappa \neq 0$. The following are equivalent:
\begin{enumerate}
\item  $(\M,\Ho,g)$ is an H-type foliation with parallel horizontal Clifford structure such that for every $Z,W \in \s{\V}$, $(\nabla_{Z} J)_{W} =J_{\Psi(Z,W)}$,
with $\Psi(Z,W)=-\kappa (Z \cdot W +\langle Z ,W \rangle )$.
\item $\V_p \subset \calC_p \left(\frac{\kappa}{2}, g_\Ho \oplus 2\kappa g_\V  \right)$, $\forall p \in \M$.
\end{enumerate}
\end{theorem}

\begin{Remark}\label{scaling inverse}
One can equivalently rewrite Theorem \ref{t:curvature constancy1} as follows. Let $(\M,\Ho,g)$ be a totally geodesic foliation with vertical distribution $\V$.  Let $K \neq 0$. The following are equivalent:
\begin{enumerate}
\item $\V_p \subset \calC_p \left(K, g  \right)$, $\forall p \in \M$.
\item  $(\M,\Ho, g_\Ho \oplus \frac{1}{4K} g_\V )$ is an H-type foliation with parallel horizontal Clifford structure for which $\kappa =2K$.
\end{enumerate}
\end{Remark}
As a preliminary for the proof of Theorem \ref{t:curvature constancy1}, we first rewrite O'Neill's formulas using the notations of this paper.

\begin{lemma}\label{ONeill}
Let $(\M,\Ho,g)$ be a totally geodesic foliation. Let us consider  the canonical variation of $g$, i.e. the one-parameter family of (semi-)Riemannian metrics defined $g_{\varepsilon}=g_\Ho \oplus  \frac{1}{\varepsilon }g_{\V},  \varepsilon \neq 0$. Let $R^{g_\ve}$ denote the Riemannian curvature of the Levi-Civita connection for $g_\ve$. Then, for every $V \in \s\V$,
$$R^{g_\ve}(V,X) Y = \left\{ \begin{array}{ll}
- \frac{1}{2} (\nabla_V T)(X,Y) - \frac{1}{2\ve} (\nabla_X J)_V Y + \frac{1}{4\ve} T(X, J_V Y ) & \text{if $X,Y \in \s\Ho$,} \\ \\
R_\V(V,X) Y &  \text{if $X,Y \in \s\V$}.
\end{array} \right.$$
\end{lemma}
\begin{proof}
We note that the Levi-Civita connection of the (semi)-Riemannian metric $g_\ve$, is given by
\begin{equation*}
\nabla_X^{g_\ve} Y = \nabla_X Y - \frac{1}{2} T(X, Y) + \frac{1}{2\ve} J_XY + \frac{1}{2\ve} J_Y X  \quad X,Y \in \s{T\M}.
\end{equation*}
We  can then either proceed by direct (but lengthy) computations or use the O'Neill's formulas (Theorem 9.28\footnote{Note that \cite{Besse} uses the opposite sign convention for the Riemannian curvature tensor} in \cite{Besse}) noting that the O'Neill's tensor $A^\ve$ of the totally geodesic foliation  $(\M,\Ho,g_\ve)$ is given by
\begin{equation*}
A^\ve_X Y=- \frac{1}{2} T(X,Y) + \frac{1}{2\ve} J_Y X, \quad X,Y \in \s{T\M}. \qedhere
\end{equation*}
\end{proof}

\begin{proof}[Proof of Theorem \ref{t:curvature constancy1}] 
$2. \implies 1$. Let $\ve=\frac{1}{2\kappa}$ and denote the inner product $g_\ve$ by $\left \langle \cdot,\cdot \right\rangle_\ve$. For $X,Y \in \s{\Ho}$ and $V,W \in \s{\V}$, one has from Lemma \ref{ONeill}
\begin{align*}
- \langle R^{g_\ve} (V,X)Y +R^{g_\ve}(V,Y)X, W \rangle_\ve &= -\frac{1}{4\ve} \langle T(X,J _VY)+T(Y,J_V X),W\rangle_\ve \\
 & =\frac{1}{4\ve^2}  \langle (J_V J_W+J_W J_V) X,Y \rangle.
\end{align*}
Using 2. we therefore obtain $J_V J_W+J_W J_V=-4\ve^2 \kappa\langle V,W\rangle_\ve \Id_\Ho=-2 \langle V,W\rangle \Id_\Ho$, which implies that $(\M,\Ho,g)$ is an  H-type foliation.  We now prove that $T$ is horizontally parallel.  From Lemma \ref{ONeill}, one has
\[
 - \frac{1}{2\ve} (\nabla_X J)_W Y =\pi_\Ho \left( R^{g_\ve}(W,X) Y \right)=0.
\]
Therefore, $(\nabla_X J)_W X =0$ which implies that $T$ is horizontally parallel.  It remains to compute $(\nabla_V J)_W$. This can be done by using once again Lemma \ref{ONeill}. Indeed,
\begin{align}
-\langle R^{g_\ve}(X,V)Y,W \rangle_\ve =-\frac{1}{2\ve} \langle (\nabla_V J)_W X, Y \rangle +\frac{1}{4\ve^2} \langle J_W X,J_V Y \rangle. \label{eq:scaling}
\end{align}
Therefore, using 2., we have
\begin{equation*}
 \langle (\nabla_V J)_W X, Y \rangle =-\tfrac{1}{2\ve} \langle J_V J_W X, Y \rangle - \kappa \langle X,Y\rangle \langle V,W\rangle ,
\end{equation*}
and the proof is complete since $\ve=\frac{1}{2\kappa}$.

$1. \implies 2$.  Let $\ve >0$. From Lemma  \ref{ONeill} we have for $X,Y \in \s{\Ho}$ and $V,W \in \s{\V}$,
\begin{align*}
\langle R^{g_\ve} (X,V)Y,W \rangle_\ve & =\frac{1}{2\ve} \langle (\nabla_V J)_W X,Y \rangle -\frac{1}{4\ve^2} \langle J_W X , J_V Y \rangle \\
&=-\frac{\kappa}{2\ve} \langle J_V J_W X,Y \rangle -\frac{\kappa}{2\ve} \langle  X,Y \rangle \langle  V,W \rangle -\frac{1}{4\ve^2} \langle J_W X , J_V Y \rangle \\
&=\left( \frac{\kappa}{2\ve}-\frac{1}{4\ve^2}\right) \langle J_W X , J_V Y \rangle-\frac{\kappa}{2\ve} \langle  X,Y \rangle \langle  V,W \rangle.
\end{align*}
Thus, if $\ve = \frac{1}{2 \kappa}$, one has
\begin{align*}
\langle R^{g_\ve} (X,V)Y,W \rangle_\ve =-\frac{1}{4\ve^2} \langle  X,Y \rangle \langle  V,W \rangle =-\frac{1}{4\ve} \langle  X,Y \rangle \langle  V,W \rangle_\ve.
\end{align*}
On the other hand, still by Lemma  \ref{ONeill}  and $\nabla_\Ho J=0$, one has $\pi_\Ho (R^{g_\ve} (X,V)Y)=0$, thus
\begin{equation*}
R^{g_\ve} (V,X)Y=\frac{1}{4\ve} \langle  X,Y \rangle_\Ho  V=\frac{\kappa}{2} \langle  X,Y \rangle_\Ho  V.
\end{equation*}
Then, from Theorem~\ref{Formula psi} and Lemma  \ref{ONeill}, we have that for $X, Y \in \s\V$, 
$$R^{g_\ve}(V, X) Y= R_\V(V, X) Y = \kappa^2 (\langle X, Y \rangle V - \langle V, Y\rangle X).$$
Using then the symmetries of the Riemannian curvature tensor $R^{g_\ve}$ and Bianchi's identity one concludes that for every $V \in \s{\V}$ and $X,Y \in \s{T\M}$,
\[
R^{g_\ve} (V,X)Y = \frac{\kappa}{2}  \left( \langle X, Y \rangle_\ve V - \langle V, Y \rangle_\ve X  \right). \qedhere
\]
\end{proof}

This theorem allows to construct many examples of H-type foliations with parallel horizontal Clifford structures coming from a submersion. In particular, we point out the following corollary:

\begin{corollary}\label{parallel even}
Let $(\M,g)$ be a Riemannian manifold that carries a rank $m+1$ parallel non flat  even Clifford structure in the sense of Moroianu-Semmelmann \cite{MS}. Then, if $n \neq 8$, the sphere bundle of this  structure is an H-type foliation with horizontal parallel Clifford structure for which $\kappa >0$.
\end{corollary}
\begin{proof}
This follows from Theorem \ref{t:curvature constancy1} and \cite[Theorem 3.6]{MS}. We note that condition (b) of \cite[Theorem 3.6]{MS} is satisfied thanks to \cite[Proposition 2.10]{MS}.
\end{proof}

We note that parallel even Clifford structures are classified in Theorem 2.14 in \cite{MS}. Because of triality, the case $n=8$ is special and  sphere bundles over 8-dimensional manifolds that carry parallel even Clifford structures do not necessarily yield  H-type foliations with horizontal parallel Clifford structure. We are now  in position to justify Table \ref{Table 1}  of the introduction. Indeed A. Moroianu and U. Semmelmann proved the following very nice result:

\begin{theorem}\cite[Theorem 3.7]{MS} 
There exists  a Riemannian submersion from a complete simply connected Riemannian manifold $\M$ to a complete simply connected Riemannian manifold $\B$ whose vertical distribution belongs to the curvature constancy of $\M$, if and only if the couple $(\M, \B)$ appears in the Table \ref{Table 1} of the introduction.
\end{theorem}

Table \ref{Table 2} of the introduction is then obtained from Table \ref{Table 1} by using the non-compact Cartan duals of the compact symmetric spaces appearing in \ref{Table 1}. Justifying Table \ref{Table 2} requires the semi-Riemannian counterpart of \cite[Theorem 3.7]{MS} which is proved in a similar way. For further details, we refer to the comments after  Theorem 3.7,  Page 965 in \cite{MS} and to the Footnote 1, Page 955 in \cite{MS}.

\subsection{Horizontal Einstein property}
In this section, we prove the following theorem:
\begin{theorem}\label{horizontal einstein}
Let $(\M, g, \Ho)$ be an H-type foliation with a parallel  horizontal Clifford structure with  $m \ge 2$  such that:
\[
\Psi (u,v) =-\kappa ( u \cdot v +\langle u, v \rangle), \quad u,v \in \V,
\]
with $\kappa \in \R$, then:
\begin{itemize}
\item If $m \neq 3$, $\ri_\Ho = \frac{\kappa}{4} \left( n + 8 (m-1) \right)g_\Ho$.
\item If $m = 3$, then at any point $\Ho$ orthogonally splits as a direct sum $\Ho^+ \oplus \Ho^-$ and for $X,Y \in \s{\Ho}$ and
\[
\ri_\Ho (X,Y) =  \frac{\kappa}{4} \Big(( n + 8) \langle  X, Y \rangle + (\dim \Ho^+-\dim \Ho^-)\langle \sigma( X) ,  Y \rangle \Big), \]
where $\sigma =\Id_{\Ho^+} \oplus (- \Id_{\Ho^-})$. 
\item If $m = 3$ and moreover $(\M, g, \Ho)$ is of quaternionic type then $\Ho^+ = \Ho, \Ho^- = \{0\}$, and thus $\ri_\Ho = \frac{\kappa}{2} \left( n + 8 \right)g_\Ho$.
\end{itemize}
\end{theorem}

\begin{remark}
For $m = 3$, $\dim \Ho^+$ and $\dim \Ho^-$ are independent of the point where they are computed. Indeed, the proof will show that $\nabla_\Ho \sigma =0$ and that both $ \Ho^+$ and $ \Ho^-$ are parallel along horizontal curves.
\end{remark}
In particular, if $m\neq 3$, then $(\M, g, \Ho)$ is always horizontally Einstein.
In the case $m=2$, the fact that $(\M, g, \Ho)$ is horizontally Einstein  is related to the fact that quaternion K\"ahler manifolds are Einstein manifolds (see Berger \cite{Ber1},  Ishihara  \cite{Ish} or Theorem 14.39 in Besse \cite{Besse}), and the algebraic structure of our  proof below somehow parallels the one of Ishihara and Besse (in the choice of a special horizontal basis).  The key lemma is the following:

\begin{lemma} \label{prop:J-symmetry}
Let $(\M, \Ho,g)$ be a totally geodesic foliation with $\nabla_\Ho T=0$. For any $X, Y \in \s\Ho$ and $Z \in \s\V$, we have
\begin{equation}
 \big[ R_\Ho(X,Y) , J_Z \big]=(\nabla_{T(X,Y)} J)_Z +J_{(\nabla_Z T)(X,Y)}.
 \end{equation}
\end{lemma}

\begin{proof}
Write the Hessian operator for $\nabla$ as $\nabla^2_{X,Y} = \nabla_X  \nabla_Y - \nabla_{\nabla_X Y}$. Using that $J$ is parallel in horizontal directions and that $R(X,Y) = \nabla^2_{X,Y}  - \nabla^2_{Y,X} + \nabla_{T(X,Y)}$, we observe that for $X,Y \in \s\Ho$ we have
\begin{equation}
R(X,Y) J = \nabla_{T(X,Y)} J.
\end{equation}
However, for $W \in \s\Ho$ and $Z\in \s\V$, we can also write
\begin{equation}
\begin{aligned}
(R(X, Y) J)_Z W &= R(X, Y) J_Z W - J_{R(X, Y) Z} W- J_Z R(X, Y) W  \\
& = R_\Ho (X, Y) J_Z W - J_{(\nabla_Z T)(X,Y)}W
- J_Z R_\Ho (X, Y) W.
\end{aligned}
\end{equation}
The result follows.
\end{proof}

We will also need the following lemma:

\begin{lemma}\label{lemma:center}
Let $(\M, \Ho,g)$ be a totally geodesic foliation with $\nabla_\Ho T=0$ and $m=3$. Let $Z_1,Z_2,Z_3$ be a local orthonormal frame of $\V$. Then  $(\M, \Ho,g)$ is of quaternionic type if and only if $J_{Z_1}J_{Z_2} J_{Z_3} \in \{ -\Id_\Ho, \Id_\Ho \}$. If $(\M, \Ho,g)$ is not of quaternionic type, then $\sigma=J_{Z_1}J_{Z_2} J_{Z_3}$ is a non-trivial horizontal isometry such that $\sigma^2=\Id_\Ho$ and that commutes with $J_{Z_1},J_{Z_2},J_{Z_3}$.
\end{lemma}

\begin{proof}
Let $p \in \M$ and $Z_1,Z_2,Z_3$ be a local vertical frame of $\V$ around $p$. Let us denote by $\fA(p) $ the algebra (for the composition of operators) generated by $J_{Z_1}, J_{Z_2},J_{Z_3}$. We note that $J_{Z_1}J_{Z_2} J_{Z_3}$ is an isometry which is in the center of $\fA(p) $.  If $\fA(p) \simeq \Quat $, then the center of $\fA(p) $ is $\R \cdot \Id_\Ho$. Therefore $J_{Z_1}J_{Z_2} J_{Z_3} \in \{ -\Id_\Ho, \Id_\Ho \}$. Conversely, if $J_{Z_1}J_{Z_2} J_{Z_3} \in \{ -\Id_\Ho, \Id_\Ho \}$, then one can check that $\{ J_z, z \in \R \oplus \V_p\}$ is an algebra and thus $\fA(p) \simeq \Quat$.
 If $(\M, \Ho,g)$ is not of quaternionic type, the statement of the lemma is immediately checked.
\end{proof}

\begin{proof}[Proof of Theorem \ref{horizontal einstein}]
Let $Z_1,\dots, Z_m$ be a local vertical orthonormal frame. We will denote $J_i=J_{Z_i}$ and for $i \neq j$, $J_{ij}=J_i J_j$. We first observe that from Lemma \ref{prop:J-symmetry} together with the parallel horizontal Clifford structure assumption, one obtains that for every $X, Y \in \s\Ho$,
\begin{align*}
 \big[ R_\Ho(X,Y) , J_i \big] & =(\nabla_{T(X,Y)} J)_{Z_i} +J_{(\nabla_{Z_i} T)(X,Y)} \\
  &=-\kappa J_{T(X,Y)\cdot Z_i + \langle T(X,Y), Z_i \rangle} +J_{(\nabla_{Z_i} T)(X,Y)}.
\end{align*}
Then, we note that
\begin{align*}
T(X,Y)\cdot Z_i + \langle T(X,Y), Z_i \rangle=-\sum_{j=1, j\neq i}^m \langle J_j X, Y \rangle Z_i \cdot Z_j,
\end{align*}
and that
\begin{align*}
J_{(\nabla_{Z_i} T)(X,Y)}&=\sum_{j=1}^m J_{\langle (\nabla_{Z_i} T)(X,Y),Z_j \rangle Z_j} =\sum_{j=1}^m J_{\langle (\nabla_{Z_i} J)_{Z_j} X, Y \rangle Z_j}  =-\kappa \sum_{j=1, j \neq i}^m \langle J_{ij} X , Y \rangle J_j.
\end{align*}
Therefore, we have
\begin{align}\label{RJ:m1}
 \big[ R_\Ho(X,Y) , J_i \big] =\kappa \sum_{j=1,j\neq i}^m \Big( \langle J_j X, Y \rangle J_{ij}  -\langle J_{ij} X , Y \rangle J_j\Big).
\end{align}

We now fix $i$, and $j \neq i$. Note that $J_i, J_j, J_{ij}$ satisfy the quaternion relations, $J_i^2=J_j^2=J_{ij}^2=J_i J_j J_{ij}=-\Id_\Ho$ and choose a local orthonormal basis \(\{X_\ell\}\) of \(\Ho\)  such that if \(X_\ell\) is in the basis, so are \(J_i X_\ell, J_j X_\ell , J_{ij} X_\ell\) (up to a $\pm$ sign). We then compute for $X, Y \in \s\Ho$,

\begin{align*}
  \ri_\Ho(X,J_iY) &= -\sum_{\ell =1}^n \left \langle R_\Ho (X,X_\ell)J_i Y , X_\ell \right\rangle \\
 &=-\sum_{\ell =1}^n \left \langle [ R_\Ho (X,X_\ell), J_i] Y , X_\ell \right\rangle -\sum_{\ell =1}^n \left \langle J_i R_\Ho (X,X_\ell) Y , X_\ell \right\rangle \\
 &=-\sum_{\ell =1}^n \left \langle [ R_\Ho (X,X_\ell), J_i] Y , X_\ell \right\rangle +\sum_{\ell =1}^n \left \langle  R_\Ho (X,X_\ell) Y , J_iX_\ell \right\rangle.
\end{align*}

On one hand, one obtains from \eqref{RJ:m1}:
\begin{align*}
\sum_{\ell =1}^n \left \langle [ R_\Ho (X,X_\ell), J_i] Y , X_\ell \right\rangle &=\kappa \sum_{\ell =1}^n \sum_{j=1,j\neq i}^m \Big( \langle J_j X, X_\ell \rangle \langle J_{ij} Y,X_\ell \rangle -\langle J_{ij} X , X_\ell \rangle \langle J_j Y, X_\ell \rangle \Big) \\
 & =\kappa  \sum_{j=1,j\neq i}^m \Big(\langle J_j X , J_{ij} Y \rangle -\langle J_{ij} X ,  J_j Y \rangle \Big) \\
 &=2 \kappa (m-1) \langle J_i X, Y \rangle.
\end{align*}
On the other hand, noticing that the set of $-J_iX_\ell \otimes X_\ell  $ and the set of  $ X_\ell \otimes J_iX_\ell $ will be identical  as \(X_\ell\) varies across the whole basis, one obtains
\begin{align*}
\sum_{\ell =1}^n \left \langle  R_\Ho (X,X_\ell) Y , J_iX_\ell \right\rangle&=\frac{1}{2} \sum_{\ell=1}^n  \Big(\left\langle  R_\Ho (X,X_\ell) Y , J_iX_\ell \right\rangle -  \left \langle  R_\Ho (X,J_iX_\ell) Y , X_\ell \right\rangle\Big) \\
&= \frac{1}{2} \sum_{\ell=1}^n \left \langle  R_\Ho (X,Y) X_\ell , J_iX_\ell \right\rangle,
\end{align*}
where the second equality follows from Bianchi's identity and symmetries of the curvature tensor. It therefore remains to compute $ \sum_{\ell=1}^n \left \langle  R_\Ho (X,Y) X_\ell , J_iX_\ell \right\rangle$. We use  the fact that the set of $ X_\ell \otimes J_iX_\ell $ and the set of $ J_j X_\ell \otimes J_{ij}X_\ell  $ will be identical  as \(X_\ell\) varies across the whole basis to obtain
\begin{align*}
2\sum_{\ell=1}^n \left \langle  R_\Ho (X,Y) X_\ell , J_iX_\ell \right\rangle &=\sum_{\ell=1}^n \left \langle  R_\Ho (X,Y) X_\ell , J_iX_\ell \right\rangle +\left \langle  R_\Ho (X,Y) J_j X_\ell , J_{ij} X_\ell \right\rangle \\
 & =\sum_{\ell=1}^n \left \langle  R_\Ho (X,Y) X_\ell , J_j J_{ij} X_\ell \right\rangle +\left \langle  R_\Ho (X,Y) J_j X_\ell , J_{ij} X_\ell \right\rangle \\
 &=\sum_{\ell=1}^n -\left \langle  J_j R_\Ho (X,Y) X_\ell ,  J_{ij} X_\ell \right\rangle +\left \langle  R_\Ho (X,Y) J_j X_\ell , J_{ij} X_\ell \right\rangle  \\
  &= \sum_{\ell=1}^n\left \langle [ R_\Ho (X,Y), J_j ]X_\ell , J_{ij} X_\ell \right\rangle.
  \end{align*}
    Now, from \eqref{RJ:m1}:
    \begin{align*}
   \sum_{\ell=1}^n\left \langle [ R_\Ho (X,Y), J_j ]X_\ell , J_{ij} X_\ell \right\rangle &=\kappa \sum_{\ell=1}^n \sum_{k=1,k \neq j}^m  \Big(\langle J_k X, Y \rangle \left\langle J_{jk}X_\ell , J_{ij} X_\ell \right\rangle  -\langle J_{jk} X , Y \rangle \left \langle J_k X_\ell , J_{ij} X_\ell \right\rangle\Big).
\end{align*}

If $k \neq i$, one has $\left\langle J_{jk}X_\ell , J_{ij} X_\ell \right\rangle=0$ and  if $k = i$, $\left\langle J_{jk}X_\ell , J_{ij} X_\ell \right\rangle=-1$. Therefore, one obtains:

\begin{align*}
 \sum_{\ell=1}^n\left \langle [ R_\Ho (X,Y), J_j ]X_\ell , J_{ij} X_\ell \right\rangle=-\kappa n  \langle J_i X, Y \rangle -\kappa \sum_{k=1,k \neq j}^m  \sum_{\ell=1}^n \langle J_{jk} X , Y \rangle \left \langle J_k X_\ell , J_{ij} X_\ell \right\rangle.
\end{align*}

The analysis of the sum $\sum_{\ell=1}^n  \left \langle J_k X_\ell , J_{ij} X_\ell \right\rangle$ will depend on $m$. If $m=2$, then one has $\sum_{\ell=1}^n  \left \langle J_k X_\ell , J_{ij} X_\ell \right\rangle=0$, because one must have $k=i$. If $m \ge 4$, then one can pick an index $s$ which is different from $i$, $j$ and $k$ so that by using invariance of the trace by a change a basis:
\begin{align*}
\sum_{\ell=1}^n  \left \langle J_k X_\ell , J_{ij} X_\ell \right\rangle &= \sum_{\ell=1}^n  \left \langle J_k J_s X_\ell , J_{ij} J_s X_\ell \right\rangle=-\sum_{\ell=1}^n  \left \langle J_k X_\ell , J_{ij} X_\ell \right\rangle.
\end{align*}
Therefore $\sum_{\ell=1}^n  \left \langle J_k X_\ell , J_{ij} X_\ell \right\rangle=0$. Summarizing the above computations, one deduces that for $i \neq j \neq k$,
\[
 \ri_\Ho(X,J_iY)=
 \begin{cases}
 -\frac{\kappa}{4} ( 8(m-1) + n )  \langle J_i X, Y \rangle, \quad {if  }\quad  m \neq 3 \\
 -\frac{\kappa}{4}  \Big( ( 16 + n )  \langle J_i X, Y \rangle + \tr_\Ho (J_{i}J_j J_k) \langle J_{jk} X , Y \rangle \Big), \quad {if  } \quad m = 3.
 \end{cases}
 \]
Therefore, substituting $Y$ by $J_iY$ one concludes
\[
 \ri_\Ho(X,Y)=
 \begin{cases}
 \frac{\kappa}{4} ( 8(m-1) + n ) \langle  X, Y \rangle, \quad {if  }\quad  m \neq 3 \\
 \frac{\kappa}{4} \Big( (16 + n ) \langle  X, Y \rangle + \tr_\Ho (J_{1}J_2 J_3) \langle J_1 J_{2} J_3 X ,  Y \rangle \Big), \quad {if  } \quad m = 3.
 \end{cases}
 \]
 By denoting $\sigma=J_1J_2J_3$,  $\Ho^+$ the 1 eigenspace of $\sigma$ and $\Ho^-$ the $-1$ eigenspace of $\sigma$, one then concludes with Lemma \ref{lemma:center}. We note that $\sigma^2=\Id_\Ho$, thus $\nabla_\Ho \sigma=0$.
\end{proof}

\subsection{Sub-Riemannian diameter and first eigenvalue estimates}

Combining Theorem \ref{horizontal einstein} with the results of Section \ref{section CD}, one obtains the following result.

\begin{corollary} \label{compact finite fundamental corollary}
Let \((\M,\Ho,g)\) be a complete H-type foliation with a parallel horizontal Clifford structure: 
$
\Psi(Z,W) = -\kappa (Z \cdot W + \langle Z,W \rangle), 
$
with \(\kappa > 0\).  Then,  \(\M\) is compact with finite fundamental group.  Moreover, 

\begin{itemize}
\item If $m\neq 3$ then its sub-Riemannian diameter is bounded above by 
$
4 \sqrt{3} \frac{\pi}{\sqrt \kappa} \sqrt{\frac{(n + 4m)(n+6m)}{n(n+8(m-1))}} ,
$
and we have the following estimate for the first eigenvalue of the sub-Laplacian
$
\lambda_1 \ge \frac{\kappa }{4} \frac{n(n+8(m-1))}{n+3m-1}. 
$
\item If $m=3$ and \((\M,\Ho,g)\) is of quaternionic type, then its sub-Riemannian diameter is bounded above by 
$
2 \sqrt{6} \frac{\pi}{\sqrt \kappa} \sqrt{\frac{(n + 12)(n+18)}{n(n+8)}} ,
$
and we have the following  estimate for the first eigenvalue of the sub-Laplacian
$
\lambda_1 \ge \frac{n\kappa }{2}. 
$
\end{itemize}
\end{corollary}

As already pointed out in Section \ref{section CD}, the diameter bounds should not expected to be sharp. However, from \cite{BK16} the eigenvalue estimates might expected to be. Indeed, consider the quaternionic Hopf fibration 
\begin{equation}
\SU(2)\hookrightarrow \Sph^{4n+3} \to \Quat P^n,
\end{equation}
on the unit sphere $(\Sph^{4n+3},g_s)$ where $g_s$ is the standard metric. Then, one has $\V_p \subset \calC_p \left(1, g_s  \right)$, $\forall p \in \Sph^{4n+3}$. Therefore, from Remark \ref{scaling inverse}, $\kappa=2$ and the above estimate  yields $\lambda_1=4n$. This is sharp, because one actually has $\lambda_1=4n$ (see \cite{BW2,Prandi15}).

\bibliographystyle{plain}
\bibliography{biblio}

\begin{thebibliography}{10}

\bibitem{ABB}
Andrei Agrachev, Davide Barilari, and Ugo Boscain.
\newblock {\em A comprehensive introduction to sub-{R}iemannian geometry},
  volume 181 of {\em Cambridge Studies in Advanced Mathematics}.
\newblock Cambridge University Press, Cambridge, 2020.

\bibitem{BRcontact}
Andrei Agrachev, Davide Barilari, and Luca Rizzi.
\newblock Sub-{R}iemannian curvature in contact geometry.
\newblock {\em J. Geom. Anal.}, 27(1):366--408, 2017.

\bibitem{Ag-Sa}
Andrei~A. Agrachev and Yuri~L. Sachkov.
\newblock {\em Control theory from the geometric viewpoint}, volume~87 of {\em
  Encyclopaedia of Mathematical Sciences}.
\newblock Springer-Verlag, Berlin, 2004.
\newblock Control Theory and Optimization, II.

\bibitem{BI17}
Davide Barilari and Stefan Ivanov.
\newblock A {B}onnet-{M}yers type theorem for quaternionic contact structures.
\newblock {\em Calc. Var. Partial Differential Equations}, 58(1):Paper No. 37,
  26, 2019.

\bibitem{BR-Popp}
Davide Barilari and Luca Rizzi.
\newblock A formula for {P}opp's volume in sub-{R}iemannian geometry.
\newblock {\em Anal. Geom. Metr. Spaces}, 1:42--57, 2013.

\bibitem{Rizzi18}
Davide Barilari and Luca Rizzi.
\newblock Sharp measure contraction property for generalized {H}-type {C}arnot
  groups.
\newblock {\em Commun. Contemp. Math.}, 20(6):1750081, 24, 2018.

\bibitem{BaudoinEMS2014}
Fabrice Baudoin.
\newblock Sub-{L}aplacians and hypoelliptic operators on totally geodesic
  {R}iemannian foliations.
\newblock In {\em Geometry, analysis and dynamics on sub-{R}iemannian
  manifolds. {V}ol. 1}, EMS Ser. Lect. Math., pages 259--321. Eur. Math. Soc.,
  Z\"urich, 2016.

\bibitem{BBG14}
Fabrice Baudoin, Michel Bonnefont, and Nicola Garofalo.
\newblock A sub-{R}iemannian curvature-dimension inequality, volume doubling
  property and the {P}oincar\'e inequality.
\newblock {\em Math. Ann.}, 358(3-4):833--860, 2014.

\bibitem{BDW18}
Fabrice Baudoin, Nizar Demni, and Jing Wang.
\newblock The horizontal heat kernel on the quaternionic anti--de {S}itter
  spaces and related twistor spaces.
\newblock {\em Potential Anal.}, 52(2):281--300, 2020.

\bibitem{BG17}
Fabrice Baudoin and Nicola Garofalo.
\newblock Curvature-dimension inequalities and {R}icci lower bounds for
  sub-{R}iemannian manifolds with transverse symmetries.
\newblock {\em J. Eur. Math. Soc. (JEMS)}, 19(1):151--219, 2017.

\bibitem{BaGr17}
Fabrice Baudoin and Erlend Grong.
\newblock Transverse {W}eitzenb\"{o}ck formulas and de {R}ham cohomology of
  totally geodesic foliations.
\newblock {\em Ann. Global Anal. Geom.}, 56(2):403--428, 2019.

\bibitem{BGKT17}
Fabrice Baudoin, Erlend Grong, Kazumasa Kuwada, and Anton Thalmaier.
\newblock Sub-{L}aplacian comparison theorems on totally geodesic {R}iemannian
  foliations.
\newblock {\em Calc. Var. Partial Differential Equations}, 58(4):Paper No. 130,
  38, 2019.

\bibitem{BGMR19}
Fabrice {Baudoin}, Erlend {Grong}, Luca {Rizzi}, and Gianmarco {Vega-Molino}.
\newblock {Comparison theorems on H-type sub-Riemannian manifolds}.
\newblock {\em arXiv e-prints}, page arXiv:1909.03532, September 2019.

\bibitem{BK16}
Fabrice Baudoin and Bumsik Kim.
\newblock The {L}ichnerowicz-{O}bata theorem on sub-{R}iemannian manifolds with
  transverse symmetries.
\newblock {\em J. Geom. Anal.}, 26(1):156--170, 2016.

\bibitem{BW1}
Fabrice Baudoin and Jing Wang.
\newblock The subelliptic heat kernel on the {CR} sphere.
\newblock {\em Math. Z.}, 275(1-2):135--150, 2013.

\bibitem{BW14}
Fabrice Baudoin and Jing Wang.
\newblock Curvature dimension inequalities and subelliptic heat kernel gradient
  bounds on contact manifolds.
\newblock {\em Potential Anal.}, 40(2):163--193, 2014.

\bibitem{BW2}
Fabrice Baudoin and Jing Wang.
\newblock The subelliptic heat kernels of the quaternionic {H}opf fibration.
\newblock {\em Potential Anal.}, 41(3):959--982, 2014.

\bibitem{Ber1}
Marcel Berger.
\newblock Sur les groupes d'holonomie homog\`ene des vari\'et\'es \`a connexion
  affine et des vari\'et\'es riemanniennes.
\newblock {\em Bull. Soc. Math. France}, 83:279--330, 1955.

\bibitem{Besse}
Arthur~L. Besse.
\newblock {\em Einstein manifolds}, volume~10 of {\em Ergebnisse der Mathematik
  und ihrer Grenzgebiete (3)}.
\newblock Springer-Verlag, Berlin, 1987.

\bibitem{MR1798609}
Charles Boyer and Krzysztof Galicki.
\newblock 3-{S}asakian manifolds.
\newblock In {\em Surveys in differential geometry: essays on {E}instein
  manifolds}, volume~6 of {\em Surv. Differ. Geom.}, pages 123--184. Int.
  Press, Boston, MA, 1999.

\bibitem{BoG}
Charles~P. Boyer and Krzysztof Galicki.
\newblock {\em Sasakian geometry}.
\newblock Oxford Mathematical Monographs. Oxford University Press, Oxford,
  2008.

\bibitem{Bad02}
Gabriel B\u{a}di\c{t}oiu and Stere Ianu\c{s}.
\newblock Semi-{R}iemannian submersions from real and complex pseudo-hyperbolic
  spaces.
\newblock {\em Differential Geom. Appl.}, 16(1):79--94, 2002.

\bibitem{Calin09}
Ovidiu Calin, Der-Chen Chang, and Irina Markina.
\newblock Geometric analysis on {$H$}-type groups related to division algebras.
\newblock {\em Math. Nachr.}, 282(1):44--68, 2009.

\bibitem{Cowling91}
Michael Cowling, Anthony~H. Dooley, Adam Kor\'{a}nyi, and Fulvio Ricci.
\newblock {$H$}-type groups and {I}wasawa decompositions.
\newblock {\em Adv. Math.}, 87(1):1--41, 1991.

\bibitem{Gray}
Alfred Gray.
\newblock Spaces of constancy of curvature operators.
\newblock {\em Proc. Amer. Math. Soc.}, 17:897--902, 1966.

\bibitem{Gromov}
Misha Gromov.
\newblock {\em Metric structures for {R}iemannian and non-{R}iemannian spaces},
  volume 152 of {\em Progress in Mathematics}.
\newblock Birkh\"{a}user Boston, Inc., Boston, MA, 1999.

\bibitem{GrTh16a}
Erlend Grong and Anton Thalmaier.
\newblock Curvature-dimension inequalities on sub-{R}iemannian manifolds
  obtained from {R}iemannian foliations: part {I}.
\newblock {\em Math. Z.}, 282(1-2):99--130, 2016.

\bibitem{Hadfield14}
Charles Hadfield.
\newblock Twistor spaces over quaternionic-k\"ahler manifolds.
\newblock {\em Master Thesis, University Pierre \& Marie Curie}, pages 1--33,
  2014.

\bibitem{Heinonen01}
Juha Heinonen.
\newblock {\em Lectures on analysis on metric spaces}.
\newblock Universitext. Springer-Verlag, New York, 2001.

\bibitem{Shanmu15}
Juha Heinonen, Pekka Koskela, Nageswari Shanmugalingam, and Jeremy~T. Tyson.
\newblock {\em Sobolev spaces on metric measure spaces}, volume~27 of {\em New
  Mathematical Monographs}.
\newblock Cambridge University Press, Cambridge, 2015.

\bibitem{Her}
Gerardo Hernandez.
\newblock On hyper {$f$}-structures.
\newblock {\em Math. Ann.}, 306(2):205--230, 1996.

\bibitem{Hladky}
Robert~K. Hladky.
\newblock Connections and curvature in sub-{R}iemannian geometry.
\newblock {\em Houston J. Math.}, 38(4):1107--1134, 2012.

\bibitem{Ish}
Shigeru Ishihara.
\newblock Quaternion {K}\"ahlerian manifolds.
\newblock {\em J. Differential Geometry}, 9:483--500, 1974.

\bibitem{Ishihara73}
Shigeru Ishihara and Mariko Konishi.
\newblock Fibred {R}iemannian space with triple of {K}illing vectors.
\newblock {\em Kodai Math. Sem. Rep.}, 25:175--189, 1973.

\bibitem{Je}
W{\l}odzimierz Jelonek.
\newblock Positive and negative {$3$}-{$K$}-contact structures.
\newblock {\em Proc. Amer. Math. Soc.}, 129(1):247--256, 2001.

\bibitem{Kaplan}
Aroldo Kaplan.
\newblock Fundamental solutions for a class of hypoelliptic {PDE} generated by
  composition of quadratic forms.
\newblock {\em Trans. Amer. Math. Soc.}, 258(1):147--153, 1980.

\bibitem{Montgomery}
Richard Montgomery.
\newblock {\em A tour of subriemannian geometries, their geodesics and
  applications}, volume~91 of {\em Mathematical Surveys and Monographs}.
\newblock American Mathematical Society, Providence, RI, 2002.

\bibitem{MS}
Andrei Moroianu and Uwe Semmelmann.
\newblock Clifford structures on {R}iemannian manifolds.
\newblock {\em Adv. Math.}, 228(2):940--967, 2011.

\bibitem{Ornea13}
Liviu Ornea, Maurizio Parton, Paolo Piccinni, and Victor Vuletescu.
\newblock Spin(9) geometry of the octonionic {H}opf fibration.
\newblock {\em Transform. Groups}, 18(3):845--864, 2013.

\bibitem{Prandi15}
Dario Prandi, Luca Rizzi, and Marcello Seri.
\newblock A sub-{R}iemannian {S}antal\'{o} formula with applications to
  isoperimetric inequalities and first {D}irichlet eigenvalue of hypoelliptic
  operators.
\newblock {\em J. Differential Geom.}, 111(2):339--379, 2019.

\bibitem{Rifford}
Ludovic Rifford.
\newblock {\em Sub-{R}iemannian geometry and optimal transport}.
\newblock Springer Briefs in Mathematics. Springer, Cham, 2014.

\bibitem{RS-3sas}
Luca Rizzi and Pavel Silveira.
\newblock Sub-riemannian ricci curvatures and universal diameter bounds for
  3-sasakian manifolds.
\newblock {\em Journal of the Institute of Mathematics of Jussieu}, pages
  1--45, 2017.

\bibitem{Sal}
Simon Salamon.
\newblock Quaternionic {K}\"ahler manifolds.
\newblock {\em Invent. Math.}, 67(1):143--171, 1982.

\bibitem{Sasaki60}
Shigeo Sasaki.
\newblock On differentiable manifolds with certain structures which are closely
  related to almost contact structure. {I}.
\newblock {\em T\^{o}hoku Math. J. (2)}, 12:459--476, 1960.

\bibitem{Strichartz}
Robert~S. Strichartz.
\newblock Sub-{R}iemannian geometry.
\newblock {\em J. Differential Geom.}, 24(2):221--263, 1986.

\bibitem{Tanno}
Shukichi Tanno.
\newblock Variational problems on contact {R}iemannian manifolds.
\newblock {\em Trans. Amer. Math. Soc.}, 314(1):349--379, 1989.

\bibitem{Tanno2}
Shukichi Tanno.
\newblock Remarks on a triple of {$K$}-contact structures.
\newblock {\em Tohoku Math. J. (2)}, 48(4):519--531, 1996.

\bibitem{Tondeur}
Philippe Tondeur.
\newblock {\em Foliations on {R}iemannian manifolds}.
\newblock Universitext. Springer-Verlag, New York.

\bibitem{W1}
Jing Wang.
\newblock The subelliptic heat kernel on the anti-de {S}itter space.
\newblock {\em Potential Anal.}, 45(4):635--653, 2016.

\end{thebibliography}

\Addresses

\end{document}